\documentclass{article}
\usepackage{fontenc}
\usepackage{amsmath}
\usepackage{amssymb}
\usepackage{amsfonts}
\usepackage{amsthm}

\newtheorem{thm}{Theorem}[section]
\newtheorem{prop}[thm]{Proposition}
\newtheorem{cor}[thm]{Corollary}

\newtheorem{lem}[thm]{Lemma}

\begin{document}

\title{Weight Changing Operators for Automorphic Forms on Grassmannians and
Differential Properties of Certain Theta Lifts}

\author{Shaul Zemel\thanks{This work was carried out while I was working at the Technical University of Darmstadt, Germany. The initial part of this work was supported by the Minerva Fellowship (Max-Planck-Gesellschaft).}}

\maketitle


\section*{Introduction}

The classical Shimura--Maa\ss\ operators $\frac{\partial}{\partial_{\tau}}+\frac{k}{2iy}$ and $y^{2}\frac{\partial}{\partial_{\overline{\tau}}}$ are well-known for taking (elliptic, real-analytic) modular forms of weight $k$ to modular forms of weight $k+2$ and $k-2$ respectively. In addition, \cite{[Ma1]} and \cite{[Ma2]} consider differential operators which have a similar effect on Siegel modular forms, a work which was generalized in \cite{[Sh2]}. The following paper \cite{[Sh3]} concerns differential operators on functions on unitary groups which have related properties. All these operators have number-theoretic as well as representation-theoretic (or Lie-algebraic) interpretations, and are therefore the subject of many research papers (see, e.g., the reference \cite{[Sh1]}, which is strongly related to the case considered in this paper, as well as \cite{[Sh4]} for some generalizations of the results of the previously mentioned references or the investigation of invariant differential operators appearing in \cite{[Sh5]}, for example).

Our first goal is to define similar operators for modular (or automorphic) forms on another type of Shimura varieties, namely quotients of Grassmannians of vector spaces of signature $(2,b_{-})$. These are obtained by interpolating the square of the Shimura--Maa\ss\ operators from the case $b_{-}=1$, the multiple Shimura--Maa\ss\ operators obtained in the case $b_{-}=2$, and the operators for Siegel modular forms appearing in the case $b_{-}=3$. One may use Lie-theoretic considerations in order to establish the existence of such operators, but obtaining their explicit formula in this way is very tedious, because of the change of coordinates between the tube domain model and the transitive free action of an appropriate parabolic subgroup of $SO^{+}(V)$. We also remark that \cite{[Sh1]} also considers differential operators on automorphic forms on orthogonal groups. However, the operators defined in that reference take automorphic forms of some weight (i.e., a representation of the maximal compact subgroup) $\rho$ to automorphic forms having weight $\rho\otimes\eta$ for some $b_{-}$-dimensional representation $\eta$, hence in particular take scalar-valued automorphic forms on Grassmannians to vector-valued functions. Moreover, since that reference works with the coordinates arising from the bounded model while we consider the tube domain model (since the explicit formulae for the theta functions are more neatly presented in this model), an appropriate change of coordinates must be employed. It is true that after this change of coordinates, using the natural bilinear form on the tangent space of the Grassmannian in the tube domain model we may indeed obtain differential operators which remain in the scalar-valued realm. Indeed, after some additional normalization we obtain the operators defined in this paper using this method. However, the calculations involved are very delicate, laborious, long, an unenlightening, for which reason we have chosen to state and prove the formulae for the operators directly.

The second goal of this paper is to present two applications of these weight
changing operators, in the theory of theta lifts. We recall the generalization,
defined in \cite{[B]}, for the Doi--Naganuma lifting first introduced in
\cite{[DN]} and \cite{[Ng]}. This map is given in \cite{[B]} in terms of a
singular theta lift, and takes weakly holomorphic elliptic modular forms to
meromorphic modular forms on Grassmannians. On the other hand, \cite{[Ze2]}
defines a similar theta lift, using the same theta functions with polynomials.
The first result of this paper states that in the case of an even dimension,
a power of our weight raising operator sends the theta lift from \cite{[Ze2]}
to the generalized Doi--Naganuma lifting of \cite{[B]}.

In addition, recall that the theta lift from Section 13 of \cite{[B]} (which is
also studied extensively in \cite{[Bru]} and others) is a \emph{real} function.
No automorphic forms of non-zero weight can be real. As a second application
for our operators we define a notion of $m$-real automorphic forms of positive
weight $m$, and show that in case one applies the theta lift from Section 14 of
\cite{[B]} (or from \cite{[Ze2]}) to a modular form with real Fourier
coefficients, then resulting theta lift is $m$-real.

\smallskip

The first half of the paper contains numerous statements whose proofs are
delayed to later sections. We choose this way of presentation since most of the
proofs consist of direct calculations, which may divert the reader's attention
from the main ideas. Specifically, the paper is divided into 4 sections. In
Section \ref{Operators} we define the weight raising and weight lowering
operators and state their properties. Section \ref{Lifts} presents the images of
certain functions under the weight raising operators, and proves the main
theorem. Section \ref{Proofswcop} presents the proofs for the assertions of
Section \ref{Operators}, while Section \ref{Proofsact} contains the missing
proofs of Section \ref{Lifts}.

\smallskip

I would like to thank J. Bruinier for numerous suggestions and intriguing
discussions regarding the results of this paper.

\section{Weight Changing Operators for Automorphic Forms on Orthogonal Groups
\label{Operators}}

In this Section we present automorphic forms on complex manifolds arising as
orthogonal Shimura varieties of signature $(2,b_{-})$, introduce the weight
raising and weight lowering operators on such forms, and give some of their
properties. The proofs of most assertions are postponed to Section
\ref{Proofswcop}.

\medskip

Let $V$ be a real vector space with a non-degenerate bilinear form of signature
$(b_{+},b_{-})$. The pairing of $x$ and $y$ in $V$ is written $(x,y)$, and
$x^{2}$ stands for the norm $(x,x)$ of $x$. For $S \subseteq V$, $S^{\perp}$
denotes the subspace of $V$ which is perpendicular to $S$. The
\emph{Grassmannian} $G(V)$ of $V$ is defined to be the set of all decompositions
of $V$ into the orthogonal direct sum of a positive definite space $v_{+}$ and a
negative definite space $v_{-}$. In the case $b_{+}=2$ (which is the only case
we consider in this paper), it is shown in Section 13 of \cite{[B]}, Sections
3.2 and 3.3 of \cite{[Bru]}, or Subsection 2.2 of \cite{[Ze2]} (among others),
that $G(V)$ carries a complex structure and has several equivalent models, which
we now briefly present. Let \[P=\big\{Z_{V}=X_{V}+iY_{V} \in
V_{\mathbb{C}}=V\otimes_{\mathbb{R}}\mathbb{C}\big|Z_{V}^{2}=0,\
(Z_{V},\overline{Z_{V}})>0\big\}.\] $Z_{V} \in V_{\mathbb{C}}$ lies in $P$ if
and only if $X_{V}$ and $Y_{V}$ are orthogonal and have the same positive norm.
$P$ has two connected components (which are interchanged by complex
conjugation), and let $P^{+}$ be one component. The map
\[P^{+} \to G(V),\quad Z_{V}\mapsto\mathbb{R}X_{V}\oplus\mathbb{R}Y_{V}\] is
surjective, and $C^{*}$ acts freely and transitively on each  fiber of this map
by multiplication. This realizes $G(V)$ as the image of $P^{+}$ in the
projective space $\mathbb{P}(L_{\mathbb{C}})$, which is an analytically open
subset of the (algebraic) quadric $Z_{V}^{2}=0$, yielding a complex structure on
$G(V)$. This is the \emph{projective model} of $G(V)$.

Let $z$ be a non-zero vector in $V$ which is \emph{isotropic}, i.e., $z^{2}=0$.
The vector space $K_{\mathbb{R}}=z^{\perp}/\mathbb{R}z$ is non-degenerate and
Lorentzian of signature $(1,b_{-}-1)$. Choosing some $\zeta \in V$ with
$(z,\zeta)=1$ and restricting the projection $z^{\perp} \to K_{\mathbb{R}}$
to $\{z,\zeta\}^{\perp}$ gives an isomorphism. We thus write $V$ as
$K_{\mathbb{R}}\times\mathbb{R}\times\mathbb{R}$, in which
\[(\alpha,a,b)=a\zeta+bz+\big(\alpha\in\{z,\zeta\}^{\perp} \cong
K_{\mathbb{R}}\big),\quad(\alpha,a,b)^{2}=\alpha^{2}+2ab+a^{2}\zeta^{2}.\] A
(holomorphic) section $s:G(V) \to P^{+}$ is defined by the pairing with $z$
being 1. Subtracting $\zeta$ from any $s$-image and taking the
$K_{\mathbb{C}}$-image of the result yields a biholomorphism between $G(V) \cong
s\big(G(V)\big)$ and the \emph{tube domain} $K_{\mathbb{R}}+iC$, where $C$ is a
cone of positive norm vectors in the Lorentzian space $K_{\mathbb{R}}$. $C$ is
called the \emph{positive cone}, and it is determined by the choice of $z$ and
the connected component $P^{+}$. The inverse biholomorphism takes $Z=X+iY \in
K_{\mathbb{C}}$ to
\[Z_{V,Z}=\bigg(Z,1,\frac{-Z^{2}-\zeta^{2}}{2}\bigg)=\bigg(X,1,\frac{Y^{2}-X^{2}
-\zeta^{2}}{2}\bigg)+i\big(Y,0,-(X,Y)\big),\] with the real and imaginary parts
denoted $X_{V,Z}$ and $Y_{V,Z}$ respectively. They are orthogonal and have norm
$Y^{2}>0$. This identifies $G(V)$ with the \emph{tube domain model}
$K_{\mathbb{R}}+iC$. Taking the other connected component of $P$ corresponds to
taking the other cone $-C$ to be the positive cone, and to the conjugate complex
structure.

\smallskip

The subgroup $O^{+}(V)$ consisting of elements of $O(V)$ preserving the
orientation on the positive definite part acts on $P^{+}$ and $G(V)$, respecting
the projection. Elements of $O(V) \setminus O^{+}(V)$ interchange the connected
components of $P$. The action of $O^{+}(V)$ (and also of the connected component
$SO^{+}(V)$) on $G(V)$ is transitive, with the stabilizer $K$ (or $SK \leq
SO^{+}(V)$) of a point being isomorphic to $SO(2) \times O(n)$ (resp. $SO(2)
\times SO(n)$). Therefore $G(V)$ is isomrphic to $O^{+}(V)/K$ and to
$SO^{+}(V)/SK$. Given an isotropic $z$ as above, the action of $O^{+}(V)$
transfers to $K_{\mathbb{R}}+iC$, and for $M \in O^{+}(V)$ and $Z \in
K_{\mathbb{R}}+iC$ we have
\[MZ_{V,Z}=J(M,Z)Z_{V,MZ},\quad\mathrm{with}\quad
J(M,Z)=(MZ_{V,Z},z)\in\mathbb{C}^{*}.\] $J$ is a \emph{factor of automorphy},
namely the equality \[J(MN,Z)=J(M,NZ)J(N,Z)\] holds for all $Z \in
K_{\mathbb{R}}+iC$ and $M$ and $N$ in $O^{+}(V)$. For such $M$ we define the
\emph{slash operator} of weight $m$, and more generally of weight $(m,n)$, by
\[\Phi[M]_{m,n}(Z)=J(M,Z)^{-m}\overline{J(M,Z)}^{-n}\Phi(MZ),\qquad[M]_{m}=[M]_{
m,0}.\] The fact that $(Z_{V},\overline{Z_{V}})=2Y^{2}$ and the definition of
$J(M,Z)$ yield the equalities
\begin{equation}
\big(\Im(MZ)\big)^{2}=\frac{Y^{2}}{|J(M,Z)|^{2}}\quad\mathrm{and}\quad\big(F(Y^{
2})^{t}\big)[M]_{m,n}=F[M]_{m+t,n+t}(Y^{2})^{t} \label{Y2mod}
\end{equation}
the latter holding for every $m$, $n$, $t$, and function $F$ on
$K_{\mathbb{R}}+iC$ (see Lemma 3.20 of \cite{[Bru]} for the first equality in
Equation \eqref{Y2mod}, and the second one follows immediately).

The invariant measure on $K_{\mathbb{R}}+iC$ is $\frac{dXdY}{(Y^{2})^{b_{-}}}$
(see Section 4.1 of \cite{[Bru]}, but one can also prove this directly, using
the generators of $O^{+}(V)$ considered in Section \ref{Proofswcop} below). Note
that this measure depends on the choice of a basis for $K_{\mathbb{R}}+iC$, but
changing the basis only multiplies this measure by a positive global scalar. Let
$\Gamma$ be a discrete subgroup $\Gamma$ of $O^{+}(V)$ of cofinite volume. In
most of the interesting cases $\Gamma$ will be either the $O^{+}$ or the
$SO^{+}$ part of the orthogonal group of an even lattice $L$ in $V$, or the
discriminant kernel of such a group. Given $m\in\mathbb{Z}$, an
\emph{automorphic form of weight $m$ with respect to $\Gamma$} is defined to be
a (complex valued) function $\Phi$ on $K_{\mathbb{R}}+iC$ for which the equation
\[\Phi(MZ)=J(M,Z)^{m}\Phi(Z),\qquad\mathrm{or\
equivalently}\qquad\Phi[M]_{m}(Z)=\Phi(Z),\] holds for all $M\in\Gamma$ and $Z
\in K_{\mathbb{R}}+iC$. Using the standard argument, such a function is
equivalent to a function on $P^{+}$ which is $-m$-homogenous (with respect to
the action of $\mathbb{C}^{*}$) and $\Gamma$-invariant, as considered, for
example, in \cite{[B]}.

\smallskip

We now consider some differential operators on functions on $K_{\mathbb{R}}+iC$.
Given a basis for $K_{\mathbb{R}}$, we write $\partial_{x_{k}}$ for
$\frac{\partial}{\partial x_{k}}$ (for $1 \leq k \leq b_{-}$). Similarly,
$\partial_{y_{k}}$ stands for the coordinates of the imaginary part from $C$.
The notation for the derivatives
$\partial_{z_{k}}=\frac{1}{2}(\partial_{x_{k}}-i\partial_{y_{k}})$ and
$\partial_{\overline{z_{k}}}=\frac{1}{2}(\partial_{x_{k}}+i\partial_{y_{k}})$
will be further shortened to $\partial_{k}$ and $\partial_{\overline{k}}$
respectively.

The operator $I=\sum_{k}x_{k}\partial_{x_{k}}$ multiplies a homogenous function
on $K_{\mathbb{R}}$ by its homogeneity degree, and is thus independent of the
choice of basis (indeed, it has an intrinsic Lie-theoretic description). The
operators
\[D^{*}=\sum_{k}y_{k}\partial_{k}\quad\mathrm{and}\quad\overline{D^{*}}=\sum_{k}
y_{k}\partial_{\overline{k}}\] from \cite{[Na]} are intrinsic as well, and they
are also invariant under translations in the real part of $K_{\mathbb{R}}+iC$.
If the basis for $K_{\mathbb{R}}$ is \emph{orthonormal}, i.e., orthogonal with
the first vector having norm 1 and the rest having norm $-1$, then the
\emph{Laplacian of $K_{\mathbb{R}}$}, denoted $\Delta_{K_{\mathbb{R}}}$, is
defined to be $\partial_{x_{1}}^{2}-\sum_{k=2}^{b_{-}}\partial_{x_{k}}^{2}$. It
is independent of the choice of the orthonormal basis (though using a basis
which is not orthonormal it takes different forms), and it is invariant under
the action of $O(K_{\mathbb{R}})$ as well as under translations in
$K_{\mathbb{R}}$. With complex coordinates it has three counterparts,
\[\Delta_{K_{\mathbb{C}}}^{h}=\partial_{1}^{2}-\sum_{k=2}^{b_{-}}\partial_{k}^{2
},\quad\Delta_{K_{\mathbb{C}}}^{\overline{h}}=\partial_{\overline{1}}^{2}-\sum_{
k=2}^{b_{-}}\partial_{\overline{k}}^{2},\quad\mathrm{and}\quad\Delta_{K_{\mathbb
{C}}}^{\mathbb{R}}=\partial_{1}\partial_{\overline{1}}-\sum_{k=2}^{b_{-}}
\partial_{k}\partial_{\overline{k}},\] which we call the \emph{holomorphic
Laplacian of $K_{\mathbb{C}}$} (of Hodge weight $(2,0)$), the
\emph{anti-holomorphic Laplacian of $K_{\mathbb{C}}$} (of Hodge weight $(0,2)$),
and the \emph{real Laplacian of $K_{\mathbb{C}}$} (of Hodge weight $(1,1)$),
respectively. These operators have the same invariance and independence
properties as $\Delta_{K_{\mathbb{R}}}$. Note that the appropriate combinations
appearing in \cite{[Bru]} and \cite{[Na]} can be identified as our operators
$\frac{1}{2}\Delta_{K_{\mathbb{C}}}^{h}$,
$\frac{1}{2}\Delta_{K_{\mathbb{C}}}^{\overline{h}}$, and
$\Delta_{K_{\mathbb{C}}}^{\mathbb{R}}$ respectively, expressed in a basis which
is not orthonormal. We shall indeed discuss and generalize the operators
$\Delta_{1}$ and $\Delta_{2}$ of \cite{[Na]} in Proposition \ref{RLcomp} below.

\smallskip

The weight changing operators and their defining property are given in
\begin{thm}
For any integer $m$ define $R_{m}^{(b_{-})}$ to be the operator
\[(Y^{2})^{\frac{b_{-}}{2}-m-1}\Delta_{K_{\mathbb{C}}}^{h}(Y^{2})^{m+1-\frac{b_{
-}}{2}}=\Delta_{K_{\mathbb{C}}}^{h}-\frac{i(2m+2-b_{-})}{Y^{2}}D^{*}-\frac{
m(2m+2-b_{-})}{2Y^{2}}.\] In
addition, define
\[L^{(b_{-})}=(Y^{2})^{2}\overline{R_{0}}=(Y^{2})^{\frac{b_{-}}{2}+1}\Delta_{K_{
\mathbb{C}}}^{\overline{h}}(Y^{2})^{1-\frac{b_{-}}{2}}=(Y^{2})^{2}\Delta_{K_{
\mathbb{C}}}^{\overline{h}}+iY^{2}(2-b_{-})\overline{D^{*}}.\] Then the
equalities
\[(R_{m}^{(b_{-})}F)[M]_{m+2}=R_{m}^{(b_{-})}\big(F[M]_{m}),\qquad(L^{(b_{-})}
F)[M]_{m-2}=L^{(b_{-})}\big(F[M]_{m})\] hold for every $\mathcal{C}^{2}$
function $F$ on $K_{\mathbb{R}}+iC$ and any $M \in O^{+}(V)$. \label{wcop}
\end{thm}
The different descriptions of $R_{m}^{(b_{-})}$ and $L^{(b_{-})}$ coincide by
Lemma \ref{LappowMov} below. Theorem \ref{wcop} has the following standard
\begin{cor}
If $\Phi$ is an automorphic form of weight $m$ on $G(V) \cong
K_{\mathbb{R}}+iC$ then $R_{m}^{(b_{-})}\Phi$ and $L^{(b_{-})}\Phi$ are
automorphic forms on $K_{\mathbb{R}}+iC$ which have weights $m+2$ and $m-2$
respectively. \label{wt+-2}
\end{cor}
In correspondence with Theorem \ref{wcop} and Corollary \ref{wt+-2} we call
$R_{m}^{(b_{-})}$ and $L^{(b_{-})}$ the \emph{weight raising operator of weight
$m$} and the \emph{weight lowering operator} for automorphic forms on
Grassmannians of signature $(2,b_{-})$ respectively. As already mentioned in the
Introduction, these operators may also be given a Lie-theoretic description (see
Section \ref{Proofswcop} for more details). However, the explicit operators
appearing in Theorem \ref{wcop} are more useful for our applications.

\medskip

We shall make use of the operator
\[D^{*}\overline{D^{*}}-\frac{\overline{D^{*}}}{2i}=\overline{D^{*}}D^{*}+\frac{
D^{*}}{2i}=\sum_{k,l}y_{k}y_{l}\partial_{k}\partial_{\overline{l}},\] which we
denote $|D^{*}|^{2}$. Lemma 2.2 of \cite{[Ze2]} shows that
\[\Delta_{m,n}^{(b_{-})}=8|D^{*}|^{2}-4Y^{2}\Delta_{K_{\mathbb{C}}}^{\mathbb{R}}
-4im\overline{D^{*}}+4inD^{*}+2n(2m-b_{-})\] is the \emph{weight $(m,n)$
Laplacian} on $K_{\mathbb{R}}+iC$, and the \emph{weight $m$ Laplacian}
$\Delta_{m}^{(b_{-})}$ is just $\Delta_{m,0}^{(b_{-})}$ (this extends the
corresponding assertion of \cite{[Na]}, since his operator $\Delta_{1}$ is our
$\Delta_{0}^{(b_{-})}$ divided by 8). The constants are normalized such that
\begin{equation}
\Delta^{(b_{-})}_{m,n}(Y^{2})^{t}=(Y^{2})^{t}\Delta^{(b_{-})}_{m+t,n+t}
\label{Y2Delta}
\end{equation}
holds for every $m$, $n$, and $t$ (see the remark after Lemma \ref{LappowMov}
below). The relations between $R_{m}^{(b_{-})}$, $L^{(b_{-})}$, and the
corresponding Laplacians are given by
\begin{prop}
The equalities
\[\Delta_{m+2}^{(b_{-})}R_{m}^{(b_{-})}-R_{m}^{(b_{-})}\Delta_{m}^{(b_{-})}=(2b_
{-}-4m-4)R_{m}^{(b_{-})}\]
and
\[\Delta_{m-2}^{(b_{-})}L^{(b_{-})}-L^{(b_{-})}\Delta_{m}^{(b_{-})}=(4m-2b_{-}
-4)L^{(b_{-})}\] hold for every $m\in\mathbb{Z}$. \label{LapRmL}
\end{prop}
We recall that an automorphic form of weight $m$ on $K_{\mathbb{R}}+iC$ is said
to have eigenvalue $\lambda$ if it is annihilated by
$\Delta_{m}^{(b_{-})}+\lambda$ (i.e., eigenvalues are of
$-\Delta_{m}^{(b_{-})}$). Hence Proposition \ref{LapRmL} has the following
\begin{cor}
If $F$ is an automorphic form of weight $m$ on $K_{\mathbb{R}}+iC$ which has
eigenvalue $\lambda$ then the automorphic forms $R_{m}^{(b_{-})}F$ and
$L^{(b_{-})}F$ have eigenvalues $\lambda+4m-2b_{-}+4$ and
$\lambda-4m+2b_{-}+4$ respectively. \label{evmov}
\end{cor}

\smallskip

By evaluating compositions of the weight changing operators one shows
\begin{prop}
The combination
\[\Xi_{m}^{(b_{-})}=(Y^{2})^{2}\Delta_{K_{\mathbb{C}}}^{h}\Delta_{K_{\mathbb{C}}
}^{\overline{h}}-iY^{2}(2m+2-b_{-})D^{*}\Delta_{K_{\mathbb{C}}}^{\overline{h}}
+iY^{2}(2-b_{-})\overline{D^{*}}\Delta_{K_{\mathbb{C}}}^{h}+\]
\[+\frac{(2-b_{-})(2m+2-b_{-})}{2}Y^{2}\Delta_{K_{\mathbb{C}}}^{\mathbb{R}}
-\frac{m(2m+2-b_{-})}{2}Y^{2}\Delta_{K_{\mathbb{C}}}^{\overline{h}}\] commutes
with all the weight $m$ slash operators as well as with the Laplacian
$\Delta_{m}^{(b_{-})}$. The commutator of the global weight raising operator and
the weight lowering operator is
\[\big[R^{(b_{-})},L^{(b_{-})}\big]_{m}=R_{m-2}^{(b_{-})}L^{(b_{-})}-L^{(b_{-})}
R_{m}^{(b_{-})}=\frac{m\Delta_{m}^{(b_{-})}}{2}-\frac{mb_{-}(2m-2-b_{-})}{4}.\]
\label{RLcomp}
\end{prop}
Proposition \ref{RLcomp} provides another proof to Lemma 2.2 of \cite{[Ze2]}
about $\Delta_{m}^{(b_{-})}$. It also implies that $\Xi_{m}^{(b_{-})}$ preserves
the spaces of automorphic forms of weight $m$ for all $m\in\mathbb{Z}$ and for
every discrete subgroup $\Gamma$ of cofinite volume in $O^{+}(V)$. It also
commutes with $\Delta_{m}^{(b_{-})}$, hence preserves eigenvalues of such
automorphic forms. By rank considerations, one can probably show that the ring
of differential operators which commute with all the slash operators of weight
$m$ is generated by $\Delta_{m}^{(b_{-})}$ and $\Xi_{m}^{(b_{-})}$, hence is
a polynomial ring in two variables (if $b_{-}>1$). This assertion should also follow from part (3) of Theorem 3.3 of \cite{[Sh5]} (since the rank of the symmetric space $G(V)$ is 2 if $b_{-}>1$), though I have not verified this in detail. As $\Delta_{0}^{(b_{-})}$ is $8\Delta_{1}$ and $\Xi_{0}^{(b_{-})}$ is $16\Delta_{2}$ in the notation of \cite{[Na]}, Proposition \ref{RLcomp} generalizes the main result of that reference to other weights. A similar argument yields results of the same sort for $(m,n)$, where a possible normalization for $\Xi_{m,n}^{(b_{-})}$ is $(Y^{2})^{-n}\Xi_{m-n}^{(b_{-})}(Y^{2})^{n}$ for which an equality similar to Equation \eqref{Y2Delta} holds. We shall not need these results in what follows.

\smallskip

We now consider compositions of the weight raising operators. The natural $l$th
power of $R_{m}^{(b_{-})}$ is the composition
\[(R_{m}^{(b_{-})})^{l}=R_{m+2l-2}^{(b_{-})}\circ\ldots \circ R_{m}^{(b_{-})}.\]
The general formulae for the resulting operator seems too complicated to write
as a combination of $\Delta_{K_{\mathbb{C}}}^{h}$, $D^{*}$, and
$\frac{1}{Y^{2}}$ with explicit coefficients. However, we can establish the
properties given in the following
\begin{prop}
$(i)$ The operator $(R_{m}^{(b_{-})})^{l}$ takes automorphic forms of weight $m$
on $G(L_{\mathbb{R}})$ to automorphic forms of weight $m+2l$. $(ii)$ In case the
former automorphic form is an eigenfunction with eigenvalue $\lambda$, the
latter is also an eigenfunction, and the corresponding eigenvalue is
$\lambda+l(4m+4l-2b_{-})$. $(iii)$ The operator $(R_{m}^{(b_{-})})^{l}$ can be
written as
\[(R_{m}^{(b_{-})})^{l}=\sum_{c=0}^{l}\sum_{a=0}^{c}A_{a,c}^{(l)}\frac{(iD^{*})^
{c-a}(\Delta_{K_{\mathbb{C}}}^{h})^{l-c}}{(-Y^{2})^{c}},\] where
$A_{0,0}^{(0)}=1$ and given the coefficients $A_{a,c}^{(l)}$ for given $l$, the
coefficient $A_{a,c}^{(l+1)}$ of the next power $l+1$ is defined recursively as
\[\sum_{s=0}^{a}\binom{c-s}{a-s}A_{s,c}^{(l)}+(2m+4l-2c+4-b_{-})\bigg(A_{a,c-1}^
{(l)}+\frac{m+2l-c+1}{2}A_{a-1,c-1}^{(l)}\bigg).\] $(iv)$ For $a=0$ the
coefficients $A_{0,c}^{(l)}$ are given by the explicit formula
\[A_{0,c}^{(l)}=\frac{l!\cdot2^{c}}{(l-c)!}\binom{m+l-\frac{b_{-}}{2}}{c}.\]
\label{Rmpowl}
\end{prop}
The binomial symbol appearing in part $(iv)$ of Proposition \ref{Rmpowl} is the
\emph{extended binomial coefficient}: Indeed, for two non-negative integers $x$
and $n$ we have \[\binom{x}{n}=\frac{1}{n!}\prod_{j=0}^{n-1}(x-j),\] a formula
which makes sense for $x\in\mathbb{R}$ (as well as $x$ in any
$\mathbb{Q}$-algebra).

Part $(i)$ of Proposition \ref{Rmpowl} follows immediately from Corollary
\ref{wt+-2}. For part $(ii)$ Corollary \ref{evmov} shows that the application of
$R_{m+2r}$ (for $0 \leq r \leq l-1$) to an eigenfunction adds $4m+8r+4-2b_{-}$
to the eigenvalue, so the assertion follows from evaluating
\[\sum_{r=0}^{l-1}(4m+8r+4-2b_{-})=l(4m+4l-2b_{-}).\] The proofs of parts
$(iii)$ and $(iv)$ are given in Section \ref{Proofswcop}.

\medskip

We recall that $M=\binom{a\ \ b}{c\ \ d} \in SL_{2}(\mathbb{R})$ defines the
holomorphic map
\[M:\tau\in\Big[\mathcal{H}=\big\{\tau=x+iy\in\mathbb{C}\big|y>0\big\}\Big]
\mapsto\frac{a\tau+b}{c\tau+d},\quad\mathrm{with}\quad j(M,\tau)=c\tau+d,\]
the latter being the factor of automorphy of this action. Modular forms of
weight $(k,l)$ (or just weight $k$ if $l=0$) with respect to a discrete subgroup
$\Gamma$ of $SL_{2}(\mathbb{R})$ with cofinite volume (with respect to the
invariant measure $\frac{dxdy}{y^{2}}$) are functions
$f:\mathcal{H}\to\mathbb{C}$ which are invariant under the corresponding weight
$(k,l)$ slash operators for elements of $\Gamma$. The weight $(k,l)$ Laplacian
is \[\Delta_{k,l}=4y^{2}\partial_{\tau}\partial_{\overline{\tau}}-2iky\partial_{
\overline{\tau}}+2ily\partial_{\tau}+l(k-1),\] normalized such that
$\Delta_{k}=\Delta_{k,0}$ annihilates holomorphic functions and the Laplacians
commute with powers of $y$ as in Equation \eqref{Y2Delta}. The
\emph{Shimura--Maa\ss\ operators}
\[\delta_{k}=y^{-k}\partial_{\tau}y^{k}=\partial_{\tau}+\frac{k}{2iy}
\quad\mathrm{and}\quad y^{2}\partial_{\overline{\tau}}\] (note the different
normalization from \cite{[Bru]} and \cite{[Ze2]}!) take modular forms of weight
$k$ to modular forms of weight $k+2$ and $k-2$ respectively, or more precisely,
satisfy an appropriate commutation relation with the slash operators for all the
elements of $SL_{2}(\mathbb{R})$. They also change Laplacian eigenvalues (again,
with respect to $-\Delta_{k}$ rather than $\Delta_{k}$): $\delta_{k}$ adds $k$
to the eigenvalue, while $y^{2}\partial_{\overline{\tau}}$ subtracts $k-2$ from
it. Moreover, the powers of the Shimura--Maa\ss\ operators are given by, e.g.,
Equation (56) in \cite{[Za]}, stating that
\[\delta_{k}^{l}=\delta_{k+2l-2}\circ\ldots\circ\delta_{k}=\sum_{r=0}^{l}\frac{
l!}{(l-r)!}\binom{k+l-1}{r}\frac{\partial_{\tau}^{l-r}}{(2iy)^{r}}\] (for
arbitrary $k$, not necessarily integral and non-negative). Theorem \ref{wcop}
and Proposition \ref{LapRmL} show that our weight changing operators
$R_{m}^{(b_{-})}$ and $L^{(b_{-})}$ have similar properties. However, our
operators are differential operators of order 2 while the Shimura--Maa\ss\
operators are of order 1. This is why the results of Propositions \ref{RLcomp}
and \ref{Rmpowl} are more complicated than the fact that
$\delta_{k-2}y^{2}\partial_{\overline{\tau}}$ is just $\frac{\Delta_{k}}{4}$,
the commutator $[\delta,y^{2}\partial_{\overline{\tau}}]_{k}$ being simply
$\frac{k}{4}$, and Equation (56) of \cite{[Za]}.

\smallskip

Nonetheless, the operators $R_{m}^{(b_{-})}$ and $L^{(b_{-})}$ for small values
of $b_{-}$ are closely related to the Shimura--Maa\ss\ operators. Indeed, for
$b_{-}=1$ the group $SO_{2,1}^{+}$ is $PSL_{2}(\mathbb{R})$ and the tube domain
$K_{\mathbb{R}}+iC$ is just $\mathcal{H}$. We have
\[J(M,\tau)=j^{2}(M,\tau),\quad\mathrm{hence}\quad
[M]_{m}=[M]_{2m}^{\mathcal{H}}\quad\mathrm{and}\quad\Delta_{m}^{(1)}=\Delta_{
2m}\] (the same assertions hold for the operators involving anti-holomorphic
weights). Our operators $R_{m}^{(1)}$ and $L^{(1)}$ are \emph{squares} of the
Shimura--Maa\ss\ operators, namely
\[R_{m}^{(1)}=\delta_{2m}^{2}=\delta_{2m+2}\delta_{2m}\quad\mathrm{and}\quad L^{
(1)}=(y^{2}\partial_{\overline{\tau}})^{2}.\] Note that in this case
\[\Xi_{m}^{(1)}=\frac{(\Delta_{2m})^{2}}{16}-\frac{m\Delta_{2m}}{8}\in\mathbb{C}
[\Delta_{m}^{(1)}=\Delta_{2m}],\] in accordance with the rank of the group being
1 rather than 2 (in particular, in the notation of \cite{[Na]} we have
$\Delta_{2}=\frac{\Delta_{1}}{4}$ in this case).

For $b_{-}>1$ many authors (including \cite{[Bru]} and \cite{[Na]}) take the
basis for $K_{\mathbb{R}}$ as two elements spanning a hyperbolic plane together
with an orthogonal basis of elements of norm $-2$. In elements of the positive
cone $C$, the first two coordinates are positive. In particular, for $b_{-}=2$
we have $K_{\mathbb{R}}+iC\cong\mathcal{H}\times\mathcal{H}$, with $\tau=x+iy$
and $\sigma=s+it$ being the two coordinates. The group $SO_{2,2}^{+}$ is an
order 2 quotient of $SL_{2}(\mathbb{R}) \times SL_{2}(\mathbb{R})$, acting on
$G(V)\cong\mathcal{H}\times\mathcal{H}$ through
\[(M,N):(\tau,\sigma)\mapsto(M\tau,N\sigma)\quad\mathrm{with}\quad
J\big((M,N),(\tau,\sigma)\big)=j(M,\tau)j(N,\sigma).\] It follows that
\[[M,N]_{m}=[M]_{m,\tau}^{\mathcal{H}}[N]_{m,\sigma}^{\mathcal{H}}\quad\mathrm{
and}\quad\Delta_{m}^{(2)}=2\Delta_{m,\tau}+2\Delta_{m,\sigma}\] (which extend to
the operators with anti-holomorphic weights as well). Our operators are
\[R_{m}^{(2)}=2\delta_{m,\tau}\delta_{m,\sigma},\quad
L^{(2)}=8y^{2}t^{2}\partial_{\overline{\tau}}\partial_{\overline{\sigma}}
\quad\mathrm{and}\quad\Xi_{m}^{(2)}=\Delta_{m,\tau}\Delta_{m,\sigma}.\] In
both cases $b_{-}=1$
and $b_{-}=2$ the assertions of this Section follow from properties of the
Shimura--Maa\ss\ operators (note that $Y^{2}$ is $2y^{2}$ for $b_{-}=1$). When
$b_{-}=2$ the special orthogonal group of a negative definite subspace is also
$SO(2)$, which makes the theory of automorphic forms more symmetric.

Working with $b_{-}=3$ in this model yields another coordinate $z=u+iv$. The
positivity of $y$, $t$, and $yt-v^{2}$ is equivalent to \[\Pi=\binom{\tau\ \
z}{z\ \ \sigma}\quad\mathrm{being\ in}\quad\mathcal{H}_{2}=\big\{\Pi=X+iY \in
M_{2}(\mathbb{C})\big|\Pi=\Pi^{t},\ Y\gg0\big\}.\] Hence $K_{\mathbb{R}}+iC$ is
identified with the Siegel upper half-plane of degree 2. The group
$SO_{2,3}^{+}$ is $PSp_{4}(\mathbb{R})$, with the symplectic action and the
factor of automorphy (hence the slash operators) from the theory of Siegel
modular forms. In this case
\[R_{m}^{(3)}=-\frac{M_{m}}{Y^{2}},\quad
L^{(3)}=-Y^{2}N_{0},\quad\mathrm{and}\quad\Delta_{m}^{(3)}=2Tr(\Omega_{m,0})\]
in the notation of \cite{[Ma1]} and \cite{[Ma2]} for degree 2 (for weight
$(m,n)$ the latter assertion extends to the modified Laplacian
$\widetilde{\Delta}_{m,n}^{(3)}$ presented in Section \ref{Proofswcop}). The
operator $\Delta_{K_{\mathbb{C}}}^{h}$ is also a constant multiple of the
operator $\mathbb{D}$ considered, for example, in \cite{[CE]} and \cite{[Ch]}.

\section{Images of Theta Lifts under $R_{m}^{(b_{-})}$ and $L^{(b_{-})}$
\label{Lifts}}

For natural $r$, $s$, $t$, and $l$ we define the polynomials
\[P_{r,s,t}(\mu,Z)=\frac{(\mu,Z_{V})^{r}(\mu,\overline{Z_{V}})^{t}}{(Y^{2})^{s}}
 \quad\mathrm{and}\quad
P_{r,s,t}^{(l)}(\mu,Z)=P_{r,s,t}(\mu,Z)(\mu_{-}^{2})^{l}.\] As a functions of
$\mu \in V$, the polynomial $P_{r,s,t}(\mu,Z)$, considered, e.g.,  in
\cite{[Ze2]}, is homogenous of degree $(r+t,0)$ with respect to the element of
$G(V)$ represented by $Z$, while $P_{r,s,t}^{(l)}(\mu,Z)$ has homogeneity degree
$(r+t,2l)$. Equation (5) of \cite{[Ze2]} extends from
$P_{r,s,t}=P_{r,s,t}^{(0)}$ to the more general polynomials $P_{r,s,t}^{(l)}$:
The equality
\begin{equation}
P_{r,s,t}^{(l)}(MZ,\mu)=j(M,Z)^{s-r}\overline{j(M,Z)}^{s-t}P_{r,s,t}^{(l)}(Z,M^{
-1}\mu) \label{Prstlmod}
\end{equation}
holds for every $\mu \in V$, $Z \in K_{\mathbb{R}}+iC$, $M \in O^{+}(V)$, and
$r$, $s$, $t$, and $l$ from $\mathbb{N}$.

We shall assume that $V=L_{\mathbb{R}}$ for some fixed even lattice $L$ (of
signature $(2,b_{-})$), and consider the theta function of $L$ which is based on
the polynomial $P_{r,s,t}^{(l)}$. These are (vector-valued) functions of
$\tau=x+iy\in\mathcal{H}$ and $Z \in K_{\mathbb{R}}+iC$, which are
sums of expressions of the form
\begin{equation}
F_{r,s,t}^{(l)}(\tau,Z,\mu)=e^{-\Delta_{v}/8\pi
y}(P_{r,s,t}^{(l)})(\mu,Z)\mathbf{e}\bigg(\tau\frac{\mu_{+}^{2}}{2}+\overline{
\tau}\frac{\mu_{-}^{2}}{2}\bigg). \label{Frstldef}
\end{equation}
Here $\mu_{\pm}$ are the parts of $\mu \in
V$ which lie in the spaces $v_{\pm}$ according to the element of $G(V)$
corresponding to $Z$, $\Delta_{v}$ is the Laplacian on $V$ which corresponds to
the \emph{majorant} of that element (i.e., to the bilinear form in which the
sign on the pairing on $v_{-}$ is inverted to be positive as well), and
$\mathbf{e}(w)=e^{2\pi iw}$ for every complex $w$. A simple and direct
calculation proves
\begin{lem}
$(i)$ We have the equality $\mu_{+}^{2}=P_{1,1,1}(\mu,Z)$. In addition, the
following equalities hold:
\[(ii)\quad\Delta_{v_{+}}P_{r,s,t}=4rtP_{r-1,s-1,t-1}.\quad(iii)\quad\Delta_{v_{
-}}(\mu_{-}^{2})^{l}=2l(2l+b_{-}-2)(\mu_{-}^{2})^{l-1}.\] \label{DeltavpmPrstl}
\end{lem}
Part $(i)$ of Lemma \ref{DeltavpmPrstl} shows that we can write the exponent in
Equation \eqref{Frstldef} as the constant
$\mathbf{e}\big(\overline{\tau}\frac{\mu^{2}}{2}\big)$ (independent of $Z$)
times $e^{-2\pi yP_{1,1,1}}$. Since the differences in the indices in part
$(ii)$ of Lemma \ref{DeltavpmPrstl} remain the same, $l$ does not affect the
weight of modularity of $P_{r,s,t}^{(l)}$, and $P_{1,1,1}$ is invariant (by
Equation \eqref{Prstlmod}), we find that replacing $P$ by $F$ in Equation
\eqref{Prstlmod} still yields a valid equation. Let $L^{*}=Hom(L,\mathbb{Z})$
be the dual lattice of $L$ and $L^{*}/L$ the (finite) \emph{discriminant group}
of $L$. Then the theta function $\Theta_{L,r,s,t}^{(l)}$ is the
$\mathbb{C}[L^{*}/L]$-valued function defined by
\[\Theta_{L,r,s,t}^{(l)}(\tau,Z)=\sum_{\gamma \in
L^{*}/L}\theta_{\gamma+L,r,s,t}^{(l)}(\tau,Z)e_{\gamma},\ \ \theta_{\gamma+L,r,
s,t}^{(l)}(\tau,Z)=\sum_{\mu\in\gamma+L}F_{r,s,t}^{(l)}(\tau,Z,\mu)\] (this
function is $\Theta_{L}(\tau,0,0;v,P_{r,s,t}^{(l)})$ in the notation of
\cite{[B]}, where $v \in G(L_{\mathbb{R}})$ corresponds to $Z \in
K_{\mathbb{R}}+iC$). The extension of Equation \eqref{Prstlmod} to
$\Theta_{L,r,s,t}^{(l)}$ shows that $\Theta_{L,r,s,t}^{(l)}$ is automorphic of
weight $(s-r,s-t)$ as a function of $Z \in K_{\mathbb{R}}+iC$, and Theorem 4.1
of \cite{[B]} shows that as a function of $\tau\in\mathcal{H}$ it is a
vector-valued modular form of weight $\big(1+r+t,2l+\frac{b_{-}}{2}\big)$ and
the \emph{Weil representation} $\rho_{L}$. The latter is a representation of the
metaplectic double cover $Mp_{2}(\mathbb{Z})$ of $SL_{2}(\mathbb{Z})$, which is
defined by sending the generators $T$ and $S$ of $Mp_{2}(\mathbb{Z})$ lying over
the elements $\binom{1\ \ 1}{0\ \ 1}$ and $\binom{0\ \ -1}{1\ \ \ \ 0}$ of
$SL_{2}(\mathbb{Z})$ respectively to
\[\rho_{L}(T)(e_{\gamma})=\mathbf{e}(\gamma^{2}/2)e_{\gamma},\]
\[\rho_{L}(S)(e_{\gamma})=\frac{\zeta_{8}^{b_{-}-b_{+}}}{\sqrt{\Delta_{L}}}\sum_
{\delta \in L^{*}/L}\mathbf{e}(-(\gamma,\delta))e_{\delta}\] respectively. For
the properties of $\rho_{L}$ see \cite{[Ze1]}, as well as the reference cited
there. The space $\mathbb{C}[L^{*}/L]$ comes with a Hermitian pairing
$\langle\cdot,\cdot\rangle_{\rho_{L}}$ in which the $e_{\gamma}$ are
orthonormal, and $\rho_{L}$ is a unitary representation with respect to this
pairing. The operation of complex conjugation on $\Theta_{L,r,s,t}^{(l)}$
interchanges $r$ and $t$ and sends $\tau$ to $-\overline{\tau}$ (this is
equivalent to multiplying the bilinear form on $V$ by $-1$, but as we rather
stay in the signature $(2,b_{-})$ setting, we prefer this anti-holomorphic
operation on $\tau$). It also replaces $\rho_{L}$ by its dual representation,
but we shall consider the effect of complex conjugation only for the automorphy
in the $Z$ variable.

\smallskip

We are interested in the action of the operators $R_{m}^{(b_{-})}$ and
$L^{(b_{-})}$ on theta kernels, and the resulting differential properties of the
associated theta lifts. Several proofs will involve comparisons of these actions
on theta kernels with the actions of the operators $\delta_{k}$ and
$y^{2}\partial_{\overline{\tau}}$ on these theta kernels (multiplied by the
appropriate powers of $y$). The latter are given (in a more general context) in
Equations (6a) and (6b) of \cite{[Ze2]}. As
$P_{r,s,t}^{(l)}=P_{r,s,t}(\mu_{-}^{2})^{l}$, Lemma \ref{DeltavpmPrstl} shows
that in our case these equations take the form
\begin{equation}
\delta_{k}y^{\frac{b_{-}}{2}+2l}\Theta_{L,r,s,t}^{(l)}=\pi
iy^{\frac{b_{-}}{2}+2l}\Theta_{L,r+1,s+1,t+1}^{(l)}+\frac{il(2l+b_{-}-2)}
{8\pi}y^{\frac{b_{-}}{2}+2l-2}\Theta_{L,r,s,t}^{(l-1)} \label{deltakTheta}
\end{equation}
(where $k=1-\frac{b_{-}}{2}+r+t-2l$) and
\begin{equation}
y^{2}\partial_{\overline{\tau}}y^{\frac{b_{-}}{2}+2l}\Theta_{L,r,s,t}^{(l)}=\pi
iy^{\frac{b_{-}}{2}+2l+2}\Theta_{L,r,s,t}^{(l+1)}+\frac{irt}{4\pi}y^{\frac{b_{-}
}{2}+2l}\Theta_{L,r-1,s-1,t-1}^{(l)} \label{lowerTheta}
\end{equation}
(note again the different normalization of these operators).

\smallskip

Recall that given a modular form $F$ of weight $1+r+t-\frac{b_{-}}{2}-2l$ and
representation $\rho_{L}$, possibly with exponential growth at the cusps, its
\emph{theta lift with respect to the polynomial $P_{r,s,t}^{(l)}$} is defined in
\cite{[B]}, \cite{[Ze2]}, and others as follows. For $w>1$ let
\[D_{w}=\big\{\tau\in\mathcal{H}\big||\Re\tau|\leq1/2,|\tau|\geq1,\Im\tau\leq
w\big\},\] and assume that
\[\lim_{w\to\infty}\int_{D_{w}}y^{1+r+t-\sigma}\langle
F(\tau),\Theta_{L}(\tau,v,p_{v})\rangle_{\rho_{L}}\frac{dxdy}{y^{2}}\] exists
for $\Re\sigma\gg0$ and defines a holomorphic function of $\sigma$ on some right
half-plane, which may be extended to a meromorphic function of $\sigma$ for all
$\sigma\in\mathbb{C}$. Then the theta lift $\Phi_{L,r,s,t}^{(l)}(F,Z)$ is the
constant term of the expansion of this meromorphic function at $\sigma=0$. Now,
the modular form $F$ has a Fourier expansion of the sort
\begin{equation}
F(\tau)=\sum_{\gamma \in
L^{*}/L}\sum_{n\in\mathbb{Q}}c_{n,\gamma}(y)q^{n}e_{\gamma}, \label{Fourier}
\end{equation}
where $q^{n}$ denotes $\mathbf{e}(n\tau)$ and the $c_{n,\gamma}$ are smooth
functions of $y=\Im\tau$, which vanish unless
$n\in\frac{\gamma^{2}}{2}+\mathbb{Z}$. The modular forms which are usually
considered also satisfies the condition that $c_{n,\gamma}=0$ unless
$n\gg-\infty$ (for $F$ which is holomorphic on $\mathcal{H}$ this means at most
a pole at the cusp, and no essential singularity).

The relations between the action of the (classical) Shimura--Maa\ss\ operators
on the lifted modular form and the action of these operators on the theta kernel
used for the theta lift are given in the following
\begin{lem}
Let $F_{\pm}$ be a modular form of weight $k\pm2=1+r+t-\frac{b_{-}}{2}-2l\pm2$
and representation $\rho_{L}$, with Fourier expansion as in Equation
\eqref{Fourier}, and assume that the the regularized theta lifts
$\Phi_{L,r,s,t}^{(l)}(Z,\delta_{k}F_{-})$ and
$\Phi_{L,r,s,t}^{(l)}(Z,y^{2}\partial_{\overline{\tau}}F_{+})$ are
well-defined. Assume that the growth condition $c_{\gamma,n}(y)=o(e^{\varepsilon
y})$ as $y\to\infty$ holds for every $\gamma$, $n$, and $\varepsilon>0$, and
that $c_{0,0}(y)$ is $o(y^{T})$ as $y\to\infty$ for some $T$. Then the theta
lift $\Phi_{L,r,s,t}^{(l)}(Z,\delta_{k}F_{-})$ coincides, up to an additive
constant which may appear only if $r=t$, with the value at $Z$ of the theta lift
of $F_{-}$ with respect to
$-y^{2}\partial_{\overline{\tau}}\Theta_{L,r,s,t}^{(l)}$. The same assertion
holds for $\Phi_{L,r,s,t}^{(l)}(Z,y^{2}\partial_{\overline{\tau}}F_{+})$ and the
theta lift of $F_{+}$ with respect to $-\delta_{k}\Theta_{L,r,s,t}^{(l)}$.
\label{trans}
\end{lem}

\begin{proof}
See Lemmas 3.4 and 3.6 of \cite{[Ze2]} as well as the argument proving Lemma 2.7
of that reference. Note the factors of $2i$ distinguishing our operators here
from those of \cite{[Ze2]}, and observe that the theta function is conjugated in
the integral defining the theta lift.
\end{proof}

The complex conjugation of $\Theta_{L,r,s,t}^{(l)}$ in the definition of the
theta lift implies that $\Phi_{L,r,s,t}^{(l)}(Z,F)$ is automorphic of weight
$(s-t,s-r)$. We shall thus consider only the case $r=s$, where the automorphy in
(the corresponding) Equation \eqref{Prstlmod} involves only $j(M,Z)$ and not its
complex conjugate. As with $P_{r,s,t}$, we may omit the superscript $(l)$ in
case $l=0$. In the same manner as in Section \ref{Operators}, we shall postpone
most of the (calculational) proofs to Section \ref{Proofsact}. Only the
assertions about theta lifts will be proved here.

\smallskip

The first assertion we are interested in is
\begin{prop}
The action of $R_{m}^{(b_{-})}$ takes
$y^{\frac{b_{-}}{2}}\overline{\Theta_{L,m,m,0}}$ to $4\pi i$ times the complex
conjugate of
$y^{2}\partial_{\overline{\tau}}\big(y^{\frac{b_{-}}{2}}\Theta_{L,m+2,m+2,0}
\big)$. \label{LSO}
\end{prop}

We remark that Proposition \ref{LSO} may be formulated in terms of comparing the
actions of elements from the universal enveloping algebras of
$\mathfrak{sl}_{2}(\mathbb{R})$ and
$\mathfrak{so}(V)\cong\mathfrak{so}_{2,b_{-}}$ on the theta kernel. However,
unlike Proposition 2.3 of \cite{[Ze2]} (and Proposition 4.5 of \cite{[Bru]}),
which compare the action of order 2 elements of both universal enveloping
algebras, here the one from the algebra of $\mathfrak{so}_{2,b_{-}}$ has order 2
while the element from $\mathfrak{sl}_{2}$ has order 1.

We can now prove establish the first property of the theta lift from
\cite{[Ze2]}.

\begin{thm}
Assume that $b_{-}$ is even, and let $f$ be a weakly holomorphic modular form of
weight $1-\frac{b_{-}}{2}-m$ and representation $\rho_{L}$. Consider the modular
form $F=\frac{1}{(2\pi i)^{m}}\delta_{1-\frac{b_{-}}{2}-m}^{m}f$, of weight
$k=1-\frac{b_{-}}{2}+m$, and its theta lift $\Phi_{L,m,m,0}(Z,F)$ considered in
Theorem 3.9 of \cite{[Ze2]}. The image of the latter automorphic form under
$\frac{1}{(8\pi^{2})^{b_{-}/2}}(R_{m}^{(b_{-})})^{b_{-}/2}$ is a meromorphic
automorphic form of weight $m+b_{-}$ on $K_{\mathbb{R}}+iC$, whose singularities
are poles of order $m+b_{-}$ along special divisors. \label{merb-even}
\end{thm}

\begin{proof}
Proposition \ref{LSO} yields the equality
\[\frac{1}{8\pi^{2}}R_{m}^{(b_{-})}y^{\frac{b_{-}}{2}}\overline{\Theta_{L,m,m,0}
}(\tau,Z)=\frac{i}{2\pi}\overline{y^{2}\partial_{\overline{\tau}}y^{\frac{b_{-}}
{ 2}}\Theta_{L,m+2,m+2,0}(\tau,Z)}\] for every $\tau\in\mathcal{H}$ and $Z \in
K_{\mathbb{R}}+iC$. As $F$ (as well as its images under any power of
$\delta_{k}$) satisfies the conditions of Lemma \ref{trans}, we establish the
equality
\[\frac{1}{8\pi^{2}}R_{m}^{(b_{-})}\Phi_{L,m,m,0}(Z,F)=\Phi_{L,m+2,m+2,0}\bigg(Z
,\frac{1}{2\pi i}\delta_{k}F\bigg).\] Repeating this argument, we get
\[\frac{1}{(8\pi^{2})^{l}}(R_{m}^{(b_{-})})^{l}\Phi_{L,m,m,0}(Z,F)=\Phi_{L,m+2l,
m+2l,0}\bigg(Z,\frac{1}{(2\pi i)^{l}}\delta_{k}^{l}F\bigg)\] for any
$l\in\mathbb{N}$. Consider now the case $l=\frac{b_{-}}{2}$. Then
$\widetilde{F}=\frac{1}{(2\pi i)^{b_{-}/2}}\delta_{k}^{b_{-}/2}F$ is
$\frac{1}{(2\pi i)^{m+b_{-}/2}}\delta_{k-2m}^{m+b_{-}/2}f$ with $f$ weakly
holomorphic of weight $1-\frac{b_{-}}{2}-m$ (which is integral since $b_{-}$ is
even). But then $\frac{1}{(2\pi i)^{m+b_{-}/2}}\delta_{k-2m}^{m+b_{-}/2}$ is
just the operator $\big(\frac{\partial_{\tau}}{2\pi i}\big)^{m+b_{-}/2}$ (which
takes $q^{n}$ from a Fourier expansion to $n^{m+b_{-}/2}q^{n}$---this is the
reason for our normalization), so that the weight $1+\frac{b_{-}}{2}+m$ modular
form $\widetilde{F}$ is again weakly holomorphic. Theorem 14.3 of \cite{[B]} now
shows that our automorphic form of weight $m+b_{-}$, which we write as
$\Phi_{L,m+b_{-},m+b_{-},0}(Z,\widetilde{F})$, is meromorphic on
$K_{\mathbb{R}}+iC$, with poles of order $m+b_{-}$ along rational quadratic
divisors associated with negative norm vectors in $L^{*}$ whose corresponding
coefficients in Equation \eqref{Fourier} do not vanish. This completes the proof
of the theorem.
\end{proof}

We remark that in case the modular form $f$ is a harmonic weak Maa\ss\ form then
the modular form $\widetilde{F}$ from the proof of Theorem \ref{merb-even} is
again weakly holomorphic. Moreover, in case the image of $f$ under the operator
$\xi_{k-2m}$ of \cite{[BF]} does not have a pole at the cusp, the theta lift
has no additional singularities, and the result of Theorem \ref{merb-even}
extends to this case. However, in the theta lift $\Phi_{L,m,m,0}(Z,F)$ itself
one can still distinguish the case where $f$ is weakly holomorphic from the one
where $F$ is such a harmonic weak Maa\ss\ form.

\medskip

For the weight lowering operator $L^{(b_{-})}$, we do not have a nice equivalent
to Proposition \ref{LSO}. However, we do have an interesting result concerning
its $m$th power. We begin with
\begin{lem}
The image of $\Theta_{L,k,n,n}^{(l)}(-\overline{\tau},Z)$ under $L^{(b_{-})}$ is
\[4\pi^{2}y^{2}\Theta_{L,k+2,n,n}^{(l+1)}(-\overline{\tau},Z)+n\bigg(2l+\frac{b_
{-}}{2}\bigg)\Theta_{L,k+1,n-1,n-1}^{(l)}(-\overline{\tau},Z)+\]
\[+\frac{n(n-1)l\big(l-1+\frac{b_{-}}{2}\big)}{4\pi^{2}y^{2}}\Theta_{L,k,n-2,n-2
}^{(l-1)}(-\overline{\tau},Z).\] \label{Ltheta}
\end{lem}

Lemma \ref{Ltheta} allows us to establish the following
\begin{prop}
For any $s\in\mathbb{N}$, the image of $\overline{\Theta_{L,m,m,0}}$ under
$(L^{(b_{-})})^{s}$ attains, on $\tau$ and $Z$, the value
\[\sum_{h}\binom{s}{h}\frac{\Gamma\big(s+\frac{b_{-}}{2}\big)}{
\Gamma\big(h+\frac{b_{-}}{2}\big)}\frac{m!(4\pi^{2}y^{2})^{h}}{(m-s+h)!}
\Theta_{L,s+h,m-s+h,m-s+h}^{(h)}(-\overline{\tau},Z).\] \label{LsTheta}
\end{prop}

The case $s=m$ in Proposition \ref{LsTheta} is of particular importance, as is
shown in the following
\begin{prop}
The expression $(L^{(b_{-})})^{m}y^{\frac{b_{-}}{2}}\overline{\Theta_{L,m,m,0}}$
equals the complex conjugate of $(-4\pi
i)^{m}\delta_{1-\frac{b_{-}}{2}-m,\tau}^{m}y^{\frac{b_{-}}{2}+2m}\Theta_{L,0,0,m
}^{(m)}(\tau,Z)$. \label{Lb-Rtaum}
\end{prop}

Automorphic forms of non-zero weight can never be real-valued, because complex
conjugation yields an automorphic form with a different weight. However,
multiplying the complex conjugate automorphic form by a power of $Y^{2}$ leads
to an object which is comparable with the image of our automorphic form under
the appropriate power of a weight changing operator, as these two functions do
have the same weight. We shall thus say that an automorphic form $\Phi$, of
positive weight $m$, is \emph{$m$-real} if its image under the $m$th power
of the weight lowering operators $L^{(b_{-})}$ coincides with its complex
conjugate multiplied by a positive multiple of $(Y^{2})^{m}$. We now show that
the theta lifts from Theorem 3.9 of \cite{[Ze2]} are $m$-real, or more
generally:

\begin{thm}
Let $F$ be as in Theorem \ref{merb-even} (but without the restriction on the
parity of $b_{-}$) , and assume that $F$ is an eigenfunction with respect to
(minus) the Laplacian of weight $1-\frac{b_{-}}{2}+m$, with eigenvalue
$\lambda=-\frac{mb_{-}}{2}$. Assume further that the Fourier coefficients
$c_{\gamma,n}$ of $F$ appearing in Equation \eqref{Fourier} are real. Then
applying the operator $(L^{(b_{-})})^{m}$ to
$\frac{i^{m}}{2}\Phi_{L,m,m,0}(Z,F)$ yields the complex conjugate of
$\frac{i^{m}}{2}\Phi_{L,m,m,0}(Z,F)$ multiplied by
$m!\Gamma\big(m+\frac{b_{-}}{2}\big)(Y^{2})^{m}/\Gamma\big(\frac{b_{-}}{2}
\big)$. \label{LmPhiconj}
\end{thm}

\begin{proof}
By Proposition \ref{Lb-Rtaum}, the image of $\frac{i^{m}}{2}\Phi_{L,m,m,0}$
under $(L^{(b_{-})})^{m}$ coincides with $\frac{i^{m}}{2}$ times the regularized
integral of $F$ paired with the function
\[(-4\pi
i)^{m}\delta_{1-\frac{b_{-}}{2}-m,\tau}^{m}y^{\frac{b_{-}}{2}+2m}\Theta_{L,0,0,m
}^{(m)}(\tau,Z).\] On the other hand, the fact that the first
index in $P_{0,0,m}$ vanishes allows us to use Equation \eqref{lowerTheta}
successively $m$ times and write \[(-\pi
i)^{m}y^{\frac{b_{-}}{2}+2m}\Theta_{L,0,0,m}^{(m)}(\tau,Z)\quad\mathrm{as\
just}\quad(-y^{2}\partial_{\overline{\tau}})^{m}y^{\frac{b_{-}}{2}}\Theta_{L,0,0
,m}(\tau,Z).\] As in the proof of Theorem \ref{merb-even}, we can write
$(L^{(b_{-})})^{m}\Phi_{L,m,m,0}(Z,F)$, using Lemma \ref{trans}, as the theta
lift
$\frac{i^{m}}{2}\Phi_{L,0,0,m}\big(Z,4^{m}\delta_{1-\frac{b_{-}}{2}-m,\tau}^{m}
(-y^{2}\partial_{\overline{\tau}})^{m}F\big)$. Now, as $F$ is an eigenfunction
and $y^{2}\partial_{\overline{\tau}}$ takes eigenfunctions to eigenfunctions, we
can replace each combination $-4\delta_{l}y^{2}\partial_{\overline{\tau}}$,
starting from the inner pair, by the appropriate eigenvalue. As after applying
$(-y^{2}\partial_{\overline{\tau}})^{r}$ the eigenvalue becomes
$\lambda-r\big(m-r-\frac{b_{-}}{2}\big)$, the modular form we plug
inside the latter lift is just $F$ multiplied by the scalar
$\prod_{r=0}^{m-1}\big[\lambda-r\big(m-r-\frac{b_{-}}{2}\big)\big]$.
Substituting the value of $\lambda$, the $r$th multiplier becomes just
$(r-m)\big(r+\frac{b_{-}}{2}\big)$, and the product is
$(-1)^{m}m!\Gamma\big(m+\frac{b_{-}}{2}\big)/\Gamma\big(\frac{b_{-}}{2}\big)$.
Division by
$m!\Gamma\big(m+\frac{b_{-}}{2}\big)(Y^{2})^{m}/\Gamma\big(\frac{b_{-}}{2}\big)$
thus gives $\frac{(-i)^{m}}{2}\Phi_{L,0,m,m}(Z,F)$, so that we need to
show why $\Phi_{L,0,m,m}(Z,F)$ is the complex conjugate of
$\Phi_{L,m,m,0}(Z,F)$. As the Fourier coefficients of $F$ are real, we obtain
$\overline{F(\tau)}=F(-\overline{\tau})$. On the other hand, we have seen that
complex conjugation on our theta function interchanges the indices $r$ and $t$
and replaces the variable $\tau$ by $-\overline{\tau}$. The required assertion
now follows from the fact that powers of $y$ and the measure
$\frac{dxdy}{y^{2}}$ are both preserved by the change of variable
$\tau\mapsto-\overline{\tau}$. This completes the proof of the theorem.
\end{proof}

We remark that the choice of $\lambda=-\frac{mb_{-}}{2}$ in Theorem
\ref{LmPhiconj} is not crucial. Any choice of $\lambda$ for which the number
$\prod_{r=0}^{m-1}\big[r\big(m-r-\frac{b_{-}}{2}\big)-\lambda\big]$ is positive
will be sufficient for Theorem \ref{LmPhiconj} to hold (with the same proof).
However, we chose this eigenvalue as it is the eigenvalue of the theta lifts
from \cite{[Ze2]}.

\section{Proofs of the Properties of $R_{m}^{(b_{-})}$ and $L^{(b_{-})}$
\label{Proofswcop}}

In this Section we include the proofs of the properties of the weight raising
and weight lowering operators appearing in Section \ref{Operators}.

\medskip

We first introduce (following \cite{[Na]}) a convenient set of generators for
$O^{+}(V)$. For $\xi \in K_{\mathbb{R}}$ we define the element $p_{\xi} \in
SO^{+}(V)$ whose action is \[\big[\mu \in
K_{\mathbb{R}}=\{z,\zeta\}^{\perp}\big]\mapsto\mu-(\mu,\xi)z,
\quad\zeta\mapsto\zeta+\xi-\frac{\xi^{2}}{2}z,\quad z \mapsto z.\] Furthermore,
given an element $A \in O(K_{\mathbb{R}})$ and a scalar $a\in\mathbb{R}^{*}$
such that $a>0$ if $A \in O^{+}(K_{\mathbb{R}})$ and $a<0$ otherwise, we let
$k_{a,A} \in O^{+}(V)$ be the element acting as \[\big[\mu \in
K_{\mathbb{R}}=\{z,\zeta\}^{\perp}\big] \mapsto
A\mu,\quad\zeta-\frac{\zeta^{2}}{2}z\mapsto\frac{1}{a}\bigg(\zeta-\frac{\zeta^{2
}}{2}z\bigg),\quad z \mapsto az.\] For any $Z \in K_{\mathbb{R}}+iC$ we have
\[p_{\xi}Z=Z+\xi,\quad J(p_{\xi},Z)=1,\quad k_{a,A}Z=aAZ,\quad\mathrm{and}\quad
J(k_{a,A},Z)=\frac{1}{a}.\] Note that the relation between $A$ and the sign of
$a$ is equivalent to preserving $C$ rather than mapping $Z$ into
$K_{\mathbb{R}}-iC$---it appears that \cite{[Na]} ignored this point. Choose now
an element of $G(K_{\mathbb{R}})$ in which the positive definite space is
generated by the norm 1 vector $u_{1}$, and consider the involution $w \in
SO^{+}(K_{\mathbb{R}})$ defined by \[\big[\mu \in
K_{\mathbb{R}}=\{z,\zeta\}^{\perp}\big]\mapsto\mu-2(\mu,u_{1})u_{1},
\quad\zeta-\frac{\zeta^{2}}{2}z\mapsto-z,\quad
z\mapsto-\bigg(\zeta-\frac{\zeta^{2}}{2}z\bigg)\] ($w$ inverts the positive
definite space $\mathbb{R}u_{1}$). Its action on $K_{\mathbb{R}}+iC$ is through
\[wZ=\frac{2}{Z^{2}}\big[Z-2(Z,u_{1})u_{1}\big]\quad\mathrm{with}\quad
J(w,Z)=\frac{Z^{2}}{2}.\] The elements $k_{a,A}$ with $(a,A)$ in the index 2
subgroup of $R^{*} \times O(K_{\mathbb{R}})$ thus defined and $p_{\xi}$ for $\xi
\in K_{\mathbb{R}}$ generate the stabilizer $St_{O^{+}(V)}(\mathbb{R}z)$ of the
isotropic space $\mathbb{R}z$ in $O^{+}(V)$ as the semi-direct product of these
groups. The fact that adding $w$ to $St_{O^{+}(V)}(\mathbb{R}z)$ generates
$O^{+}(V)$ is now easily verified by considering the action on isotropic
1-dimensional subspaces of $V$.

Some useful relations are derived in the following
\begin{lem}
Let $K_{\mathbb{R}}$ be a non-degenerate vector space of dimension $b_{-}$, fix
$\alpha\in\mathbb{C}$, and let $F$ be a $\mathcal{C}^{2}$ function which is
defined on a neighborhood of a point $Z=X+iY \in K_{\mathbb{C}}$ with $Y^{2}>0$.
Then the following relations hold:
\[(Y^{2})^{-\alpha}\Delta_{K_{\mathbb{C}}}^{h}\big((Y^{2})^{\alpha}
F\big)(Z)=\Delta_{K_{\mathbb{C}}}^{h}F(Z)-\frac{2i\alpha}{Y^{2}}D^{*}F(Z)-\frac{
\alpha(\alpha-1+\frac{b_{-}}{2})}{Y^{2}}F(Z)\] and
\[(Y^{2})^{-\alpha}\Delta_{K_{\mathbb{C}}}^{\overline{h}}\big((Y^{2})^{\alpha}
F\big)(Z)=\Delta_{K_{\mathbb{C}}}^{\overline{h}}F(Z)+\frac{2i\alpha}{Y^{2}}
\overline{D^{*}}F(Z)-\frac{\alpha(\alpha-1+\frac{b_{-}}{2})}{Y^{2}}F(Z).\]
\label{LappowMov}
\end{lem}
We remark that Lemma \ref{LappowMov} holds for $K_{\mathbb{R}}$ of arbitrary
signature (not necessarily Lorentzian), but not negative definite (for $Y^{2}>0$
to be possible).

\begin{proof}
The proof is obtained by a straightforward calculation, using an orthonormal
basis for $K_{\mathbb{R}}$ and the action of $\partial_{k}$ and
$\partial_{\overline{k}}$ on functions of $Y$ alone.
\end{proof}

We remark that the third operator $\Delta_{K_{\mathbb{C}}}^{\mathbb{R}}$ bears a
property similar to Lemma \ref{LappowMov}, which is used implicitly in Section 3
of \cite{[Ze2]} in order to prove Equation \eqref{Y2Delta}.

We can now present the
\begin{proof}[Proof of Theorem \ref{wcop}]
Multiply both sides of the desired assertion for $R_{m}^{(b_{-})}$, as well as
the function $F$ there, by $(Y^{2})^{m}$. Lemma \ref{LappowMov}, the first
definition of $R_{m}^{(b_{-})}$, and Equation \eqref{Y2mod} show that this
yields the equivalent equality
\[(R_{0}^{(b_{-})}F)[M]_{2,-m}=R_{0}^{(b_{-})}\big(F[M]_{0,-m}).\]
Observe that conjugating the latter equation and multiplying by $(Y^{2})^{2}$
yields the required equality for $L^{(b_{-})}$. Hence we are reduced to proving
only this equality. Moreover, $R_{0}^{(b_{-})}$ involves only holomorphic
differentiations, which means that it commutes with the power of
$\overline{J(M,Z)}$ coming from the anti-holomorphic weights. Hence we can take
$m=0$, which implies that proving the equation
\[(R_{0}^{(b_{-})}F)[M]_{2}=R_{0}^{(b_{-})}\big(F[M]_{0})\] (which the assertion
for $R_{0}^{(b_{-})}$ in the formulation of the theorem) suffices for proving
the theorem. Writing the arguments as $M^{-1}(Z)$ in both sides and using the
cocycle condition brings the latter equation to the form
\begin{equation}
(R_{0}^{(b_{-})}F)(Z)J(M^{-1},Z)^{2}=(R_{0}^{(b_{-})})^{M^{-1}}F(Z).
\label{R0red}
\end{equation}
By a standard argument it suffices to verify Equation \eqref{R0red} for $M^{-1}$
being one of the generators of $O^{+}(V)$ considered above. Equation
\eqref{R0red} with $M^{-1}=p_{\xi}$ follows from the invariance of both
$\Delta_{K_{\mathbb{C}}}^{h}$ and $D^{*}$ under translations of $X=\Re Z$ and
the fact that $J(p_{\xi},Z)=1$. The action of $M^{-1}=k_{a,A}$ divides
$\Delta_{K_{\mathbb{C}}}^{h}$ by $a^{2}$, leaves $D^{*}$ invariant, and divides
$Y^{2}$ by $a^{2}$ (since $A \in O(K_{\mathbb{R}})$), which proves Equation
\eqref{R0red} since $J(k_{a,A},Z)=\frac{1}{a}$. Finally, for $M^{-1}=w$ we have
the equalities
\[(\Delta_{K_{\mathbb{C}}}^{h})^{w}=\bigg(\frac{Z^{2}}{2}\bigg)^{2}\Delta_{K_{
\mathbb{C}}}^{h}-(b_{-}-2)\frac{Z^{2}}{2}D,\qquad(D^{*})^{w}=\frac{Z^{2}}{
\overline{Z}^{2}}D^{*}-\frac{2iY^{2}}{\overline{Z}^{2}}D\] with
$D=\sum_{k}z_{k}\partial_{k}$ from \cite{[Na]} (the corresponding operator from
\cite{[Na]} is $\frac{1}{2}\Delta_{K_{\mathbb{C}}}^{h}$ rather than
$\Delta_{K_{\mathbb{C}}}^{h}$, while $\delta=\frac{Z^{2}}{2}$,
$\overline{\delta}=\frac{\overline{Z}^{2}}{2}$, and $d=\frac{Y^{2}}{2}$ there).
Using Equation \eqref{Y2mod} we thus find that applying $M^{-1}=w$ to the sum of
$\Delta_{K_{\mathbb{C}}}^{h}$ and $\frac{i(b_{-}-2)}{Y^{2}}D^{*}$ (which is
$R_{0}^{(b_{-})}$) multiplies it by $\big(\frac{Z^{2}}{2}\big)^{2}$ (as the
coefficients in front of $D$ cancel), which establishes Equation \eqref{R0red}
also for this case using the value of $J(w,Z)$. This completes the proof of the
theorem.
\end{proof}

In order to indicate what is the Lie-theoretic interpretation of the operators
$R_{m}^{(b_{-})}$ and $L^{(b_{-})}$, we recall the vector $u_{1}$ we used for
defining $w$ above, and take a vector $\tilde{u} \in K_{\mathbb{R}}$ of norm
$-1$ which is orthogonal to $u_{1}$ (we assume here $b_{-}>1$, but for $b_{-}=1$
our operators are squares of the order 1 operators $\delta_{2m}$ and
$y^{2}\partial_{\overline{\tau}}$, whose Lie-theoretic interpretation is given,
e.g., in \cite{[Ve]}). These choices determine the parabolic subgroup of
$SO^{+}(V)$ appearing in the following
\begin{prop}
Let $H_{K_{\mathbb{R}}}$ be the subgroup of $SO^{+}(K_{\mathbb{R}})$ consisting
of those matrices which preserve the isotropic subspace
$\mathbb{R}(u_{1}+\tilde{u})$ and whose action on the quotient
$(u_{1}+\tilde{u})^{\perp}/\mathbb{R}(u_{1}+\tilde{u})$ is trivial. Define $H$
to be the group generated by all the elements $p_{\xi}$ with $\xi \in
K_{\mathbb{R}}$ and by the elements $k_{a,A}$ with $a>0$ and $A \in
H_{K_{\mathbb{R}}}$. Then the group $H$ operates freely and transitively on
$K_{\mathbb{R}}+iC$. \label{parab}
\end{prop}
Let $K \cong SO(2) \times SO(b_{-})$ be the stabilizer, in $SO^{+}(V)$, of the
element of $G(V)$ represented by $Z=iu_{1}$, and let $\mathfrak{k}$ be its Lie
algebra. The action of a normalized generator of
$\mathfrak{so}(2)\subseteq\mathfrak{k}$ on $\mathfrak{so}(V)_{\mathbb{C}}$
decomposes the latter space into the eigenspaces with eigenvalue 0 (this is
precisely $\mathfrak{k}$) and $\pm i$ (complex conjugate spaces of dimension
$b_{-}$ each). Hence the action on the space of products of two elements of
$\mathfrak{so}(V)$ (inside its universal enveloping algebra, say) decomposes
into eigenspaces with eigenvalues 0 and $\pm2i$. One verifies that in each of
the $\pm2i$-eigenspaces, precisely one combination commutes with the part
$\mathfrak{so}(b_{-})$ of $\mathfrak{k}$. As our automorphic forms correspond to
functions on $SO(V)$ on which $SO(2) \subseteq K$ operates according to a
specific character and $SO(b_{-})$ operate trivially (normalized suitably),
these elements (of order 2) of the universal enveloping algebra of
$\mathfrak{so}(V)$ lead to weight raising and weight lowering operators. One may
then evaluate, using the interplay between the operations of $\mathfrak{k}$ and
the Lie algebra of the group $H$ from Proposition \ref{parab}, the action of
these operators, and find that they lead to our $R_{m}^{(b_{-})}$ and
$L^{(b_{-})}$. However, the change of coordinates between $H_{K_{\mathbb{R}}}$ and $K_{\mathbb{R}}+iC$ in this evaluation is more tedious than one might have believe.

\smallskip

We also indicate briefly the connection between our operators and those of \cite{[Sh1]}. That reference defines, for every representation $\rho$ of $\mathbb{C}^{\times} \times GL_{b_{-}}(\mathbb{C})$ (a subgroup of which we identify the complexification of the compact subgroup $K$, which is isomorphic to the product $\mathbb{C}^{\times} \times SO(b_{-},\mathbb{C})$), a differential operator that roughly sends (vector-valued) automorphic forms with weight (i.e., representation) $\rho$ to automorphic forms having representation $\rho\otimes\omega$, where $\omega$ is the standard representation of that product on $\mathbb{C}^{b_{-}}$. This representation space is considered as the holomorphic cotangent space of $G(V)$, and the operator is, in fact, just the holomorphic differential map $d$, twisted by the image of a scalar $\eta$ and a matrix $\xi$ (both defined explicitly in \cite{[Sh1]}) via the representation $\rho$. Starting with the 1-dimensional representation which is the $m$th power of $\mathbb{C}^{\times}$ (this is the representation associated with our automorphic forms of weight $m$) and repeating this operation twice, we obtain an automorphic form with representation involving $\omega^{\otimes2}$. The idea is expressing the resulting automorphic form when $\omega$ is identified with $K_{\mathbb{C}}$, and using the bilinear form on the latter space in order to replace the $\omega^{\otimes2}$-valued automorphic forms by scalar-valued ones.

Now, we replace the coordinate denoted $z$ in \cite{[Sh1]} by $u=\sqrt{2}z$, considering it as lying in the complexified space $(v_{-})_{\mathbb{C}}$ associated to some base point for $G(V)$, and decompose it as some multiple $u_{z}$ of $z_{v_{-}}$ plus a vector $u_{\perp}$ which is perpendicular to $z_{v_{-}}$. Here $z$ is again the isotropic vector we used for defining $K_{\mathbb{R}}$. Choosing the positive part of $z$ appropriately (recall that the vector denoted $p(z)$ in \cite{[Sh1]} is not presented in the canonical form), we obtain that our norm 0 vector has pairing $1+u^{t}u-2u^{t}z_{v_{-}}$ with $z$ and its positive and negative $K_{\mathbb{C}}$ coordinates are $i(1-u^{t}u)$ and $2u_{\perp}$ respectively. It follows that the associated element $Z$ of $K_{\mathbb{C}}$ (which can be shown to be in $K_{\mathbb{R}}+iC$) satisfies $(Z+ie_{+})^{2}=\frac{-4}{1+u^{t}u-2u^{t}z_{-}}$ (where $e_{+}$ is the generator of the positive part of $K_{\mathbb{R}}$), so that the inverse map sends $Z$ to the vector obtained by multiplying the positive part of $-2\frac{Z+ie_{+}}{(Z+ie_{+})^{2}}$ by $i$, and adding $z_{v_{-}}$ to the result. Given an automorphic form $F$ of weight $m$ on $G(V)$, a very lengthy, tedious, and involved calculation gives us the expression for the $\omega^{\otimes2}$-valued automorphic form obtained from $F$ under the operator mentioned in the previous paragraph, and after applying the pairing we obtain an expression closely related to $(Z+ie_{+})^{2m}R_{m}^{b_{-}}[(Z+ie_{+})^{-2m}F]$. Indeed, the expression denoted by $\eta$ in \cite{[Sh1]} becomes $\frac{16Y^{2}}{|(Z+ie_{+})^{2}|^{2}}$ using our variable, so that multiplying by $\eta^{m}$ before applying the operator and by $\eta^{-m}$ afterwards corresponds to the operation involving $Y^{2m}$ appearing in the definition of $R_{m}^{b_{-}}$, as well as the additional operation with $(Z+ie_{+})^{2m}$. However, the details of this calculation are very long as well, and therefore we have chosen to state and prove Theorem \ref{wcop} more directly.

\smallskip

For calculational purposes it turns out convenient to introduce the operator
\[\widetilde{\Delta}_{m,n}^{(b_{-})}=\Delta_{m,n}^{(b_{-})}-2n(2m-b_{-}),\]
on which complex conjugation interchanges the indices $m$ and $n$. The operator
\[(D^{*})^{2}-\frac{D^{*}}{2i}=\sum_{k,l}y_{k}y_{l}\partial_{k}\partial_{l}\]
will also show up, so we denote it $\widetilde{(D^{*})^{2}}$. We now turn to the
\begin{proof}[Proof of Proposition \ref{LapRmL}]
Conjugating the desired equality for $R_{m}^{(b_{-})}$ by $(Y^{2})^{m}$,
applying Equation \eqref{Y2Delta}, and taking the differences between the
operators $\widetilde{\Delta}_{m,n}^{(b_{-})}$ and $\Delta_{m,n}^{(b_{-})}$ into
consideration, we see that the asserted equality for $R_{m}^{(b_{-})}$ is
equivalent to
\[\widetilde{\Delta}_{2,-m}^{(b_{-})}R_{0}^{(b_{-})}-R_{0}^{(b_{-})}\widetilde{
\Delta}_{0,m}^{(b_{-})}=(2b_{-}+4m-4)R_{0}^{(b_{-})}.\] Moreover, multiplying
the complex conjugate of the latter equation by $(Y^{2})^{2}$ and comparing
$\widetilde{\Delta}_{2,-m}^{(b_{-})}$ with $\Delta_{2,-m}^{(b_{-})}$ yields the
required property for $L^{(b_{-})}$ (with the index $m$ replaced by $-m$).
Hence, as in the proof of Theorem \ref{wcop}, we are reduced to proving this
single equation. In addition, the dependence on $m$ of the left hand side enters
only through the difference $-4imD^{*}$ between the operators
$\widetilde{\Delta}_{l,-m}^{(b_{-})}$ and $\Delta_{l}^{(b_{-})}$ with
$l\in\{0,2\}$. As a simple calculation yields
\[\quad\big[D^{*},\Delta_{K_{\mathbb{C}}}^{h}\big]=i\Delta_{K_{\mathbb{C}}}^{h}
\quad\mathrm{and}\quad\bigg[D^{*},\frac{D^{*}}{Y^{2}}\bigg]=\frac{iD^{*}}{Y^{2}
},\] it suffices to prove the equality for $m=0$ (i.e., the original assertion
for $R_{0}^{(b_{-})}$):
\[\Delta_{2}^{(b_{-})}R_{0}^{(b_{-})}-R_{0}^{(b_{-})}\Delta_{0}^{(b_{-})}=(2b_{-
}-4)R_{0}^{(b_{-})}.\] The commutator of $\Delta_{0}^{(b_{-})}$ and
$R_{0}^{(b_{-})}$ is evaluated using the equalities
\[\big[|D^{*}|^{2},\Delta_{K_{\mathbb{C}}}^{h}\big]=i\overline{D^{*}}\Delta_{K_{
\mathbb{C}}}^{h}+iD^{*}\Delta_{K_{\mathbb{C}}}^{\mathbb{R}}+\frac{\Delta_{K_{
\mathbb{C}}}^{\mathbb{R}}}{2},\]
\[\bigg[|D^{*}|^{2},\frac{D^{*}}{Y^{2}}\bigg]=\frac{3i|D^{*}|^{2}-i\widetilde{
(D^{*})^{2}}+D^{*}}{2Y^{2}},\quad\big[Y^{2}\Delta_{K_{\mathbb{C}}}^{\mathbb{R}},
\Delta_{K_{\mathbb{C}}}^{h}\big]=2iD^{*}\Delta_{K_{\mathbb{C}}}^{\mathbb{R}}
+\frac{b_{-}}{2}\Delta_{K_{\mathbb{C} }}^{\mathbb{R}},\]
\[\mathrm{and}\quad\bigg[Y^{2}\Delta_{K_{\mathbb{C}}}^{\mathbb{R}},\frac{D^{*}}{
Y^{2}}\bigg]=\frac{2i|D^{*}|^{2}-2i\widetilde{(D^{*})^{2}}+i\Delta_{K_{\mathbb{C
}}}^{h}+i\Delta_{K_{\mathbb{C}}}^{\mathbb{R}}+(2-b_{-})D^{*}}{2Y^{2}}\] (which
all follow from straightforward calculations). Applying the equalities
\[\Delta_{2}^{(b_{-})}=\Delta_{0}^{(b_{-})}-8i\overline{D^{*}}\quad\mathrm{and}
\quad\overline{D^{*}}\circ\big(\frac{D^{*}}{Y^{2}}\big)=\frac{2|D^{*}|^{2}-iD^{*
}}{2Y^{2}}\] and putting in the appropriate scalars now establishes the
proposition.
\end{proof}

Our next task is the
\begin{proof}[Proof of Proposition \ref{RLcomp}]
We begin by evaluating $R_{m-2}^{(b_{-})}L^{(b_{-})}$ written as
\[R_{m-2}^{(b_{-})}(Y^{2})^{2}\Delta_{K_{\mathbb{C}}}^{\overline{h}}+R_{m-2}^{
(b_{-})}i(2-b_{-})Y^{2}\overline{D^{*}}=(Y^{2})^{2}R_{m}^{(b_{-})}\Delta_{K_{
\mathbb{C}}}^{\overline{h}}+i(2-b_{-})Y^{2}R_{m-1}^{(b_{-})}\overline{D^{*}}.\]
Using the equalities
\[\big[\Delta_{K_{\mathbb{C}}}^{h},\overline{D^{*}}\big]=-i\Delta_{K_{\mathbb{C}
}}^{\mathbb{R}}\quad\mathrm{and}\quad
D^{*}\overline{D^{*}}=|D^{*}|^{2}+\frac{\overline{D^{*}}}{2i}\] we establish the
equation
\[R_{m-2}^{(b_{-})}L^{(b_{-})}=\Xi_{m}^{(b_{-})}+\frac{(2-b_{-})(2m-b_{-})}{8}
\Delta_{m}^{(b_{-})},\] where $\Xi_{m}^{(b_{-})}$ is defined in the formulation
of the proposition. We now decompose $R_{m}^{(b_{-})}$ in
$L^{(b_{-})}R_{m}^{(b_{-})}$ (which is
$(Y^{2})^{2}\overline{R_{0}^{(b_{-})}}R_{m}^{(b_{-})}$), yielding
\[(Y^{2})^{2}\overline{R_{0}^{(b_{-})}}\Delta_{K_{\mathbb{C}}}^{h}-i(2m+2-b_{-}
)Y^{2}\overline{R_{-1}^{(b_{-})}}D^{*}-\frac{m(2m+2-b_{-})}{2}Y^{2}\overline{R_{
-1}^{(b_{-})}}.\] The formulae
\[\big[\Delta_{K_{\mathbb{C}}}^{\overline{h}},D^{*}\big]=i\Delta_{K_{\mathbb{C}}
}^{\mathbb{R}}\quad\mathrm{and}\quad\overline{D^{*}}D^{*}=|D^{*}|^{2}-\frac{D^{*
}}{2i}\] now show that
\[L^{(b_{-})}R_{m}^{(b_{-})}=\Xi_{m}^{(b_{-})}-\frac{b_{-}(2m+2-b_{-})}{8}
\Delta_{m}^{(b_{-})}+\frac{mb_{-}(2m+2-b_{-})}{4}.\] The required commutation
relation follows. As Theorem \ref{wcop} shows that the compositions
$R_{m-2}^{(b_{-})}L^{(b_{-})}$ and $L^{(b_{-})}R_{m}^{(b_{-})}$ commute with all
the slash operators of weight $m$, and Proposition \ref{LapRmL} implies
that these operators commute with $\Delta_{m}$, the assertion about
$\Xi_{m}^{(b_{-})}$ is also established. This proves the proposition.
\end{proof}

Finally, we come to the
\begin{proof}[Proof of parts $(iii)$ and $(iv)$ of Proposition \ref{Rmpowl}]
We prove part $(iii)$ by induction (the case $l=0$ being trivial). If
$(R_{m}^{(b_{-})})^{l}$ is presented by the asserted formula then
$(R_{m}^{(b_{-})})^{l+1}$, which is $R_{m+2l}^{(b_{-})}(R_{m}^{(b_{-})})^{l}$,
equals
\[R_{m+2l}^{(b_{-})}\sum_{c=0}^{l}\sum_{s=0}^{c}A_{s,c}^{(l)}\frac{(iD^{*})^{c-s
}(\Delta_{K_{\mathbb{C}}}^{h})^{l-c}}{(-Y^{2})^{c}}=\sum_{s,c}A_{s,c}^{(l)}\frac
{R_{m+2l-c}^{(b_{-})}(iD^{*})^{c-s}(\Delta_{K_{\mathbb{C}}}^{h})^{l-c}}{(-Y^{2}
)^{c}}.\] For each $c$, the term involving
$\frac{D^{*}}{Y^{2}}$ (resp. $\frac{1}{Y^{2}}$) in $R_{m+2l-c}^{(b_{-})}$ takes
the term with indices $c$ and $s$ (for $l$) to a multiple of the term with
corresponding to $c+1$ and $s$ (resp. $c+1$ and $s+1$) for $l+1$. For
$\Delta_{K_{\mathbb{C}}}^{h}$ we have
\[\big[\Delta_{K_{\mathbb{C}}}^{h},iD^{*}\big]=\Delta_{K_{\mathbb{C}}}^{h}
\quad\mathrm{hence}\quad\Delta_{K_{\mathbb{C}}}^{h}(iD^{*})^{c-s}=\sum_{a=s}^{c}
\binom{c-s}{a-s}(iD^{*})^{c-a}\Delta_{K_{\mathbb{C}}}^{h},\] and we multiply the
latter sum by $\frac{(\Delta_{K_{\mathbb{C}}}^{h})^{l-c}}{(-Y^{2})^{c}}$. This
shows that $(R_{m}^{(b_{-})})^{l+1}$ can be expressed by the asserted formula.
Putting in the multipliers $A_{s,c}^{(l)}$ from $(R_{m}^{(b_{-})})^{l}$ and the
coefficients of $\frac{D^{*}}{Y^{2}}$ and $\frac{1}{Y^{2}}$ in
$R_{m+2l-c}^{(b_{-})}$, summing over $c$ and $s$, and taking the coefficient in
front of the term with indices $c$ and $a$ (and $l+1$) in the result, we obtain
the recursive relation asserted in part $(iii)$. We now observe that for $a=0$
the recursive formula reduces to
\[A_{0,c}^{(l+1)}=A_{0,c}^{(l)}+(2m+4l-2c+4-b_{-})A_{0,c-1}^{(l)}.\] Denote the
asserted value of $A_{0,c}^{(l)}$ by $B_{0,c}^{(l)}$. As
$A_{0,0}^{(0)}=1=B_{0,0}^{(0)}$, it suffices to show that the numbers
$B_{0,c}^{(l)}$ satisfy the latter recursive formula. But the equality
\[2(l-c+1)\bigg(m+l-c-\frac{b_{-}}{2}+1\bigg)+c(2m+4l-2c+4-b_{-}
)=2(l+1)\bigg(m+l-\frac{b_{-}}{2}+1\bigg)\] holds for every $l$ and $c$ (and
$m$ and $b_{-}$), and multiplication by $\frac{l!\cdot2^{c-1}}{c(l+1-c)!}$ and
by the binomial coefficient $\binom{m+l-\frac{b_{-}}{2}}{c-1}$ yields the
required recursive relation for the numbers $B_{0,c}^{(l)}$. This completes the
proof of the proposition.
\end{proof}

\section{Actions on Theta Kernels---Proofs \label{Proofsact}}

The main technical lemma, which will be required for the evaluations in most of
the following proofs, is based on
\begin{lem}
Given $\mu \in L_{\mathbb{R}}$, the operators $R_{0}^{(b_{-})}$ and
$L^{(b_{-})}$ take the function $P_{1,1,1}$ of $Z \in K_{\mathbb{R}}+iC$ to
$-\frac{b_{-}}{2}P_{0,2,2}$ and $-\frac{b_{-}}{2}P_{2,0,0}$. \label{RmLP111}
\end{lem}

\begin{proof}
The commutation relation between powers of $Y^{2}$ and the operators
$R_{m}^{(b_{-})}$ obtained from the first definition of the latter operators in
Theorem \ref{wcop} and the fact that the latter operators involve only
holomorphic differentiation allows us to write $R_{0}^{(b_{-})}P_{1,1,1}$ as
$P_{0,1,1}R_{-1}^{(b_{-})}(\mu,Z_{V,Z})$. Hence we must evaluate the operation
of $\Delta_{K_{\mathbb{C}}}^{h}$ and $D^{*}$ on $(\mu,Z_{V,Z})$. For the latter
operator a simple calculation yields
\[2iD^{*}(\mu,Z_{V,Z})=2i(\mu,Y_{V,Z})+2Y^{2}(\mu,z)=(\mu,Z_{V,Z})-(\mu,
\overline{Z_{V,Z}})+2Y^{2}(\mu,z).\] The former operator is pure of weight 2,
hence its action gives a non-zero result only on the part
$-\frac{Z^{2}}{2}(\mu,z)$, and using an orthonormal basis one finds that this
result is just $-b_{-}(\mu,z)$. Combining these results, we find that
\[\bigg[R_{-1}^{(m)}=\Delta_{K_{\mathbb{C}}}^{h}+\frac{ib_{-}}{Y^{2}}D^{*}-\frac
{b_{-}}{2Y^{2}}\bigg](\mu,Z_{V,Z})=-\frac{b_{-}}{2Y^{2}}(\mu,\overline{Z_{V,Z}})
,\] from which the value of $R_{0}^{(b_{-})}P_{1,1,1}$ follows. The assertion
about $L^{(b_{-})}P_{1,1,1}$ is a consequence of the value of
$R_{0}^{(b_{-})}P_{1,1,1}$, since $P_{1,1,1}$ is a real function and
$L^{(b_{-})}$ is the operator which is complex conjugate to $R_{0}^{(b_{-})}$,
multiplied by $(Y^{2})^{2}$. This proves the lemma.
\end{proof}

Another useful evaluation appears in the following
\begin{lem}
The holomorphic and anti-holomorphic $Z$-gradients of $P_{1,1,1}$ have, as
vectors in $K_{\mathbb{C}}$, the norms $P_{0,2,2}\mu_{-}^{2}$ and
$P_{2,2,0}\mu_{-}^{2}$ respectively. \label{gradnorm}
\end{lem}

\begin{proof}
$(\mu,\overline{Z_{V,Z}})$ is anti-holomorphic, and the holomorphic gradients
of $(\mu,Z_{V,Z})$ and $Y^{2}$ are $\mu_{K_{\mathbb{R}}}-(\mu,z)Z$ and $-iY$
respectively, where $\mu_{K_{\mathbb{R}}}$ is the orthogonal projection of $\mu
\in L_{\mathbb{R}}$ onto $K_{\mathbb{R}}=\{z,\zeta\}^{\perp}$. It follows that
$P_{1,1,1}$ has holomorphic gradient
\[P_{0,2,1}\big[Y^{2}(\mu_{K_{\mathbb{R}}}-(\mu,z)Z)+i(\mu,Z_{V,Z})Y\big].\]
Now, the (easily evaluated) equalities
\[\big(\mu_{K_{\mathbb{R}}}-(\mu,z)Z,Y\big)=(\mu,Y_{V,Z})-iY^{2}(\mu,z)\] and
\[(\mu,z)^{2}Z^{2}-2(\mu,z)(\mu_{K_{\mathbb{R}}},Z)+2(\mu,z)(\mu,Z_{V,Z})=2(\mu,
z)(\mu,\zeta)-\zeta^{2}(\mu,z)^{2}\] reduce to the norm of the latter gradient
\[P_{0,2,2}\big[\mu_{K_{\mathbb{R}}}^{2}+2(\mu,\zeta)(\mu,z)-\zeta^{2}(\mu,z)^{2
}-P_{1,1,1}\big].\] But $\mu$ is
$\big(\mu_{K_{\mathbb{R}}},\mu_{z},(\mu,\zeta)-\zeta^{2}\mu_{z}\big)$ in the
$K_{\mathbb{R}}\times\mathbb{R}\times\mathbb{R}$ coordinates, so that the sum
of the first three terms in the brackets is just $\mu^{2}$. Subtracting
$P_{1,1,1}=\mu_{+}^{2}$ completes the proof of the first assertion, and the
second assertion follows from complex conjugation since the function
$P_{1,1,1}$ is real-valued. This proves the lemma.
\end{proof}

\smallskip

For $\mu \in L_{\mathbb{R}}$ and $\tau=x+iy\in\mathcal{H}$ we denote the vector
$\sqrt{2\pi y}\mu$ by $\tilde{\mu}$. Its norm is $2\pi y\mu^{2}$, and after
choosing an element of $G(L_{\mathbb{R}})$, it decomposes into $\tilde{\mu}_{+}$
(of norm $2\pi y\mu_{+}^{2}$) and $\tilde{\mu}_{-}$ (whose norm is $2\pi
y\mu_{-}^{2}$). We now prove

\begin{prop}
Let $f:\mathbb{R}\to\mathbb{R}$ be a smooth function. Then the images of the
function $f(\tilde{\mu}_{+})$ under $R_{0}^{(b_{-})}$ and $L^{(b_{-})}$ are
$2\pi
yP_{0,2,2}\big[\tilde{\mu}_{-}^{2}f''(\tilde{\mu}_{+})-\frac{b_{-}}{2}f'(\tilde{
\mu}_{+})\big]$ and $2\pi
yP_{2,0,0}\big[\tilde{\mu}_{-}^{2}f''(\tilde{\mu}_{+})-\frac{b_{-}}{2}f'(\tilde{
\mu}_{+})\big]$ respectively. \label{1st}
\end{prop}

\begin{proof}
Both operators consist of a first order operator $D$ (a multiple of $D^{*}$ or
of $\overline{D^{*}}$) and a second order operator $\Delta$ (which equals
$\Delta_{K_{\mathbb{C}}}^{h}$ or $\Delta_{K_{\mathbb{C}}}^{\overline{h}}$).
Then $D\big(f(T)\big)=DT \cdot f'(T)$, and $\Delta\big(f(T)\big)$ is the sum of
$\Delta T \cdot f'(T)$ and an expression involving $f''(T)$. In our case
$T=\tilde{\mu}_{+}=2\pi y\mu_{+}^{2}=2\pi yP_{1,1,1}$, so that the coefficient
of $f'(T)$ is just $2\pi y$ times $R_{0}^{(b_{-})}P_{1,1,1}$ and
$L^{(b_{-})}P_{1,1,1}$, and the latter expressions are evaluated using Lemma
\ref{RmLP111}. The coefficients of $f''(T)$ coming from $\Delta$ being
$\Delta_{K_{\mathbb{C}}}^{h}$ or
$(Y^{2})^{2}\Delta_{K_{\mathbb{C}}}^{\overline{h}}$ are
the norms (in $K_{\mathbb{C}}$) of the holomorphic and anti-holomorphic
gradients of $T$, the latter being multiplied by $(Y^{2})^{2}$. For $T=2\pi
yP_{1,1,1}$ these norms take the values given in Lemma \ref{gradnorm},
multiplied by $(2\pi y)^{2}$. Gathering these results together and substituting
the value of $\tilde{\mu}_{-}^{2}$ completes the proof of the proposition.
\end{proof}

\smallskip

We now turn to proving assertions concerning the images of theta lifts (or
complex conjugates of theta functions), having only holomorphic weights of
automorphy, under the operators $R_{m}^{(b_{-})}$ and $L^{(b_{-})}$. This was
seen to boil down to the operation on the function $F_{r,s,t}^{(l)}$ from
Equation \eqref{Frstldef}, with $\tau$ replaced by $-\overline{\tau}$, under the
additional assumption $s=t$. The exponent was seen, using part (i) of Lemma
\ref{DeltavpmPrstl} to be
$\mathbf{e}\big(-\tau\frac{\mu^{2}}{2}\big)e^{-\tilde{\mu}_{+}^{2}}$, where the
first multiplier is a constant (i.e., independent of $Z$). The polynomial part
is evaluated in
\begin{lem}
$(i)$ For any natural numbers $k$ and $n$ we have \[(2\pi
y)^{n}P_{n-k,0,0}e^{-\Delta_{v_{+}}/8\pi y}(P_{k,n,n})e^{-2\pi
yP_{1,1,1}}=(-1)^{k}\frac{d^{k}}{dT^{k}}(T^{n}e^{-T})\bigg|_{T=\tilde{\mu}_{+}^{
2}}.\] $(ii)$ Applying $e^{\Delta_{v_{-}}/8\pi y}$ to $(2\pi
y)^{l}(\mu_{-}^{2})^{l}$ yields
$\sum_{p}\binom{l}{p}\big[\Gamma\big(l+\frac{b_{-}}{2}\big)/
\Gamma\big(p+\frac{b_{-}}{2}\big)\big]\big(\tilde{\mu}_{-}^{2}\big)^{p}$.
\label{expT}
\end{lem}

We allow the index $n-k$ appearing in Part (i) here to be negative, with the
natural extension of the definition of $P_{r,s,t}$ to negative $r$. We remark
that the expressions obtained in this part are just the generalized Laguerre
polynomials $L_{k}^{(n-k)}$, multiplied by the exponents, and normalized
appropriately.

\begin{proof}
Multiple applications of part (ii) of Lemma \ref{DeltavpmPrstl} show that
\[\frac{\Delta_{v_{+}}^{h}}{h!(-8\pi
y)^{h}}P_{k,n,n}=\frac{k!n!P_{k-h,n-h,n-h}}{(k-h)!(n-h)!h!(-2\pi y)^{h}}.\]
Multiplying by $(2\pi y)^{n}P_{n-k,0,0}$ and summing over $h$, the left hand
side of the equation in part $(i)$ becomes just
\[\sum_{h}\binom{k}{h}\frac{n!}{(n-h)!}(-1)^{h}(2\pi yP_{1,1,1})^{n-h}e^{-2\pi
yP_{1,1,1}}.\] On the other hand, differentiating the product $T^{n}e^{-T}$ $k$
times with respect to $T$ yields
\[\sum_{h=0}^{k}\binom{k}{h}\bigg(\frac{d}{dT}\bigg)^{k}T^{n}\cdot\bigg(\frac{d}
{dT}\bigg)^{k-h}e^{-T}=\sum_{h=0}^{k}\binom{k}{h}\frac{n!T^{n-h}}{(n-h)!}(-1)^{
k-h}e^{-T},\] and substituting $T=\tilde{\mu}_{+}^{2}=2\pi yP_{1,1,1}$ yields
the same expression multiplied by $(-1)^{k}$. This establishes part $(i)$. For
part $(ii)$, applying part (iii) of Lemma \ref{DeltavpmPrstl} successively
evaluates
\[\Delta_{v_{-}}^{l-p}(\mu_{-}^{2})^{l}=\frac{4^{l-p}l!}{p!}\cdot\frac{
\Gamma\big(l+\frac{b_{-}}{2}\big)}{\Gamma\big(p+\frac{b_{-}}{2}\big)}(\mu_{-}^{2
})^{p}.\] Dividing this term by $(8\pi y)^{l-p}(l-p)!$, multiplying everything
by $(2\pi y)^{l}$, and substituting $\tilde{\mu}_{-}^{2}=2\pi y\mu_{-}^{2}$
gives the asserted expression. This completes the proof of the lemma.
\end{proof}

As $\tilde{\mu}_{-}^{2}=\tilde{\mu}^{2}-\tilde{\mu}_{+}^{2}$, Lemma \ref{expT}
implies that the dependence of the expression $(2\pi
y)^{n+l}P_{n-k,0,0}F_{k,n,n}^{(l)}(-\overline{\tau},Z,\mu)$ (or the
corresponding theta function) on the variable $Z$ is only through the quantity
$\tilde{\mu}_{+}^{2}$. For convenience, we gather these results in the
following
\begin{cor}
Define the functions
\[f_{k,n,p}^{(w)}(T)=(-1)^{k}\frac{d^{k}}{dT^{k}}(T^{n}e^{-T})\cdot(w-T)^{p},\]
where $k$, $p$, and $n$ are natural numbers and $w\in\mathbb{R}$. Then the theta
function $\Theta_{L,k,n,n}^{(l)}(-\overline{\tau},Z)$ equals \[\sum_{\mu \in
L^{*}}\sum_{p}\binom{l}{p}\frac{\Gamma\big(l+\frac{b_{-}}{2}\big)}{
\Gamma\big(p+\frac{b_{-}}{2}\big)}\frac{f_{k,n,p}^{(\tilde{\mu}^{2})}(\tilde{\mu
}_{+}^{2})}{(2\pi
y)^{n+l}P_{n-k,0,0}}\mathbf{e}\bigg(-\tau\frac{\mu^{2}}{2}\bigg)e_{\mu+L}.\]
\label{eDeltaPmu}
\end{cor}

\begin{proof}
Just substitute the value of $e^{-\Delta_{v}/8\pi y}(P_{k,n,n}^{(l)})$, which
equals the product of $e^{-\Delta_{v_{+}}/8\pi y}(P_{k,n,n})$ and
$e^{-\Delta_{v_{+}}/8\pi y}\big((\mu_{-}^{2})^{l}\big)$, from Lemma \ref{expT}
into the expression defining the theta function.
\end{proof}

\smallskip

We can now present the
\begin{proof}[Proof of Proposition \ref{LSO}]
As seen above, it suffices to consider the action of $R_{m}^{(b_{-})}$ only on
the expression $P_{0,m,m}(\mu,Z)e^{-2\pi yP_{1,1,1}}$ with fixed $\mu$ (recall
that $P_{0,m,m}$ is harmonic). The holomorphicity of the differentiation in
$R_{m}^{(b_{-})}$ shows that the result is the same as
$P_{0,m,m}R_{0}^{(b_{-})}e^{-\tilde{\mu}_{+}^{2}}$. By putting $f(T)=e^{-T}$,
Proposition \ref{1st} evaluates $R_{0}^{(b_{-})}e^{-\tilde{\mu}_{+}^{2}}$ as
$2\pi yP_{0,2,2}\big(\tilde{\mu}_{-}+\frac{b_{-}}{2}\big)e^{-\tilde{\mu}_{+}}$,
and multiplying by $P_{0,m,m}$ yields
\[R_{m}^{(b_{-})}P_{0,m,m}e^{-\tilde{\mu}_{+}^{2}}=4\pi^{2}y^{2}P_{0,m+2,m+2}
\bigg[\mu_{-}^{2}+\frac{b_{-}}{4\pi y}\bigg]e^{-2\pi yP_{1,1,1}}.\] But the
expression in parentheses is $e^{\Delta_{v_{-}}/8\pi y}(\mu_{-}^{2})$ by part
$(ii)$ of Lemma \ref{expT}, and the harmonicity of $P_{0,m+2,m+2}$ allows us to
put it also into the action of $e^{-\Delta_{v}/8\pi y}$ without affecting the
resulting expression. Putting in the missing constant
$y^{\frac{b_{-}}{2}}\mathbf{e}\big(-\tau\frac{\mu^{2}}{2}\big)e_{\mu+L }$ and
summing over $\mu \in L^{*}$ we establish the equality
\[R_{m}^{(b_{-})}y^{\frac{b_{-}}{2}}\Theta_{L,0,m,m}(-\overline{\tau},Z)=4\pi^{2
}y^{
2+\frac{b_{-}}{2}}\Theta_{L,0,m+2,m+2}^{(1)}(-\overline{\tau},Z).\] But as
$P_{m+2,m+2,0}$ is harmonic, Equation \eqref{lowerTheta} shows that applying the
operator $-4\pi iy^{2}\partial_{\overline{\tau}}$ to
$y^{\frac{b_{-}}{2}}\Theta_{L,m+2,m+2,0}(\tau,Z)$ yields the complex conjugate
of the latter expression, and complex conjugation inverts the sign of $4\pi i$.
This proves the proposition.
\end{proof}

\smallskip

We now turn to the
\begin{proof}[Proof of Lemma \ref{Ltheta}]
Write the theta function $\Theta_{L,k,n,n}^{(l)}(-\overline{\tau},Z)$ as in
Corollary \ref{eDeltaPmu}. It suffices to fix $\mu \in L^{*}$ and compare the
coefficients of $\mathbf{e}\big(-\tau\frac{\mu^{2}}{2}\big)e_{\mu+L}$ in both
sides. Take some $0 \leq p \leq l$, and apply Proposition \ref{1st} with the
function $f=f_{k,n,p}^{(\tilde{\mu}^{2})}$. The powers of $2\pi y$ and
$P_{1,0,0}$ from Corollary \ref{eDeltaPmu} and Proposition \ref{1st} merge to
$(2\pi y)^{n+l-1}P_{n-2-k,0,0}$ in the denominator, and the remaining part of
$L^{(b_{-})}f_{k,n,p}^{(\tilde{\mu}^{2})}$ is
\[\binom{l}{p}\frac{\Gamma\big(l+\frac{b_{-}}{2}\big)}{\Gamma\big(p+\frac{b_{-}}
{2}\big)}\bigg[\tilde{\mu}_{-}^{2}\big(f_{k,n,p}^{(\tilde{\mu}^{2})}
\big)''(\tilde{\mu}_{+}^{2})-\frac{b_{-}}{2}\big(f_{k,n,p}^{(\tilde{\mu}^{2})}
\big)'(\tilde{\mu}_{+}^{2})\bigg].\] As
$\tilde{\mu}_{-}^{2}=\tilde{\mu}^{2}-\tilde{\mu}_{+}^{2}$, and as one easily
evaluates \[(f_{k,n,p}^{w})'(T)=-pf_{k,n,p-1}^{w}(T)-f_{k+1,n,p}^{w}(T),\] the
part in brackets in latter expression equals
\begin{equation}
f_{k+2,n,p+1}^{(\tilde{\mu}^{2})}(\tilde{\mu}_{+}^{2})+\bigg(2p+\frac{b_{-}}{2
}\bigg)f_{k+1,n,p}^{(\tilde{\mu}^{2})}(\tilde{\mu}_{+}^{2})+p\bigg(p+\frac{b_{-}
}{2}-1\bigg)f_{k,n,p-1}^{(\tilde{\mu}^{2})}(\tilde{\mu}_{+}^{2}). \label{f'dec}
\end{equation}
We now write the denominator in the preceding constant as
\[\frac{\big(p+\frac{b_{-}}{2}\big)}{\Gamma\big(p+1+\frac{b_{-}}{2}\big)},
\quad\frac{1}{\Gamma\big(p+\frac{b_{-}}{2}\big)},\quad\mathrm{and}\quad\frac{1}{
\big(p-1+\frac{b_{-}}{2}\big)\Gamma\big(p-1+\frac{b_{-}}{2}\big)}\] in front of
the three terms in Equation \eqref{f'dec} respectively, and after taking the sum
over $p$ and gathering the functions with the same index $p$ together, we see
that the quotient
$\Gamma\big(l+\frac{b_{-}}{2}\big)/\Gamma\big(p+\frac{b_{-}}{2}\big)$
multiplies
\[\bigg(p-1+\frac{b_{-}}{2}\bigg)\binom{l}{p-1}f_{k+2,n,p}^{(\tilde{\mu}^{2})}
+\bigg(2p+\frac{b_{-}}{2}\bigg)\binom{l}{p}f_{k+1,n,p}^{(\tilde{\mu}^{2})}
+(p+1)\binom{l}{p+1}f_{k,n,p}^{(\tilde{\mu}^{2})}\] (where we have omitted the
variable $\tilde{\mu}_{+}^{2}$). Using the identity
$b\binom{a}{b}=a\binom{a-1}{b-1}$ we can write the latter expression as
\[l\bigg[\binom{l-1}{p-2}f_{k+2,n,p}^{(\tilde{\mu}^{2})}(\tilde{\mu}_{+}^{2}
)+2\binom{l-1}{p-1}f_{k+1,n,p}^{(\tilde{\mu}^{2})}(\tilde{\mu}_{+}^{2})+\binom{
l-1}{p}f_{k,n,p}^{(\tilde{\mu}^{2})}(\tilde{\mu}_{+}^{2})\bigg]+\]
\begin{equation}
+\frac{b_{-}}{2}\bigg[\binom{l}{p-1}f_{k+2,n,p}^{(\tilde{\mu}^{2})}(\tilde{\mu}_
{+}^{2})+\binom{l}{p}f_{k+1,n,p}^{(\tilde{\mu}^{2})}(\tilde{\mu}_{+}^{2})\bigg].
\label{fknpwcomb}
\end{equation}
Now, differentiating $k$ times and multiplying by $(w-T)^{p}$ takes the equality
\[(T^{n}e^{-T})'=(nT^{n-1}-T^{n})e^{-T}\quad\mathrm{to}\quad f_{k,n,p}^{(w)}
(T)-f_{k+1,n,p}^{(w)}(T)=nf_{k,n-1,p}^{(w)}(T).\] One application of this
relation replaces $f_{k+1,n,p}^{(\tilde{\mu}^{2})}$ by
$f_{k+2,n,p}^{(\tilde{\mu}^{2})}+nf_{k+1,n-1,p}^{(\tilde{\mu}^{2})}$, and we
also obtain
\[f_{k,n,p}^{(\tilde{\mu}^{2})}=f_{k+2,n,p}^{(\tilde{\mu}^{2})}+2nf_{k+1,n-1,p}^
{(\tilde{\mu}^{2})}+n(n-1)f_{k,n-2,p}^{(\tilde{\mu}^{2})}.\] Each of the terms
in Equation \eqref{fknpwcomb} thus contributes to the total coefficient in front
of $f_{k+2,n,p}^{(\tilde{\mu}^{2})}$, which using the classical properties of
the binomial coefficients reduces to
$\big(l+\frac{b_{-}}{2}\big)\binom{l+1}{p}$. Using the recursive property of the
gamma function again, we obtain the coefficient
$\binom{l+1}{p}\Gamma\big(l+1+\frac{b_{-}}{2}\big)/\Gamma\big(p+\frac{b_{-}}{2}
\big)$, which together with \[\frac{1}{P_{n-2-k,0,0}(2\pi
y)^{n+l-1}}=\frac{4\pi^{2}y^{2}}{P_{n-2-k,0,0}(2\pi y)^{n+l+1}}\] yields the
coefficient appearing in front of
$f_{k+2,n,p}^{(\tilde{\mu}^{2})}(\tilde{\mu}_{+}^{2})$ in the expansion of
$4\pi^{2}y^{2}\Theta_{L,k+2,n,n}^{(l+1)}(-\overline{\tau},Z)$ in Corollary
\ref{eDeltaPmu}. The total coefficient in front of the function
$f_{k+1,n-1,p}^{(\tilde{\mu}^{2})}$ in Equation \eqref{fknpwcomb} becomes
(again, using binomial identities) just
$\big(2l+\frac{b_{-}}{2}\big)\binom{l}{p}$, and the gamma quotient and the
powers of $2pi y$ and $P_{1,0,0}$ complete the formula for the second asserted
term. For the remaining term
$n(n-1)l\binom{l-1}{p}f_{k,n-2,p}^{(\tilde{\mu}^{2})}$ from Equation
\eqref{fknpwcomb} we use the functional equation of the gamma function again to
write $\Gamma\big(l+\frac{b_{-}}{2}\big)$ as
$\big(l-1+\frac{b_{-}}{2}\big)\Gamma\big(l-1+\frac{b_{-}}{2}\big)$, and we also
decompose \[P_{n-2-k,0,0}(2\pi y)^{n+l-1}=4\pi^{2}y^{2}P_{n-2-k,0,0}(2\pi
y)^{n-2+l-1}.\] Corollary \ref{eDeltaPmu} then establishes the remaining
asserted term in a similar manner. This completes the proof of the lemma.
\end{proof}

We go on to the
\begin{proof}[Proof of Proposition \ref{LsTheta}]
We prove the assertion by induction on $s$. The case $s=0$ is trivial. Denote
the asserted coefficient corresponding to the $h$th term in the expression for
the image under $(L^{(b_{-})})^{s}$ by $a_{s,h}(y)$. We need to evaluate
\[\sum_{h}a_{s,h}(y)L^{(b_{-})}\Theta_{L,s+h,m+s-h,m+s-h}^{(h)},\] and compare
it with the asserted expression for $s+1$. Lemma \ref{Ltheta} shows that for
each $h$ the $L^{(b_{-})}$-image of the corresponding theta function is a linear
 combination of three theta functions, which correspond to the index $s+1$ and
the indices $h-1$, $h$, and $h+1$. After applying the appropriate summation
index changes, the coefficient which we get in front of
$\Theta_{L,s+1+h,m-s-1+h,m-s-1+h}^{(h)}$ in
$(L^{(b_{-})})^{s+1}\overline{\Theta_{L,m,m,0}}$ is
\[4\pi^{2}y^{2}a_{s,h-1}(y)+(m-s+h)\bigg(2h+\frac{b_{-}}{2}\bigg)a_{s,h}(y)+\]
\[+\frac{(m-s+h)(m-s+h+1)(h+1)\big(h+\frac{b_{-}}{2}\big)}{4\pi^{2}y^{2}}a_{s,
h+1}(y).\] Substituting the values of $a_{s,t}$ for $t$ being $h-1$, $h$, and
$h+1$, one easily sees that all three terms yield the same multiplier
$\frac{m!(4\pi^{2}y^{2})^{h}}{(m-s-1+h)!}$. Applying the functional equation
for the gamma function in the first and third term, we obtain that the
remaining expression equals
\[\frac{\Gamma\big(s+\frac{b_{-}}{2}\big)}{\Gamma\big(h+\frac{b_{-}}{2}\big)}
\bigg[\bigg(h-1+\frac{b_{-}}{2}\bigg)\binom{s}{h-1}+\bigg(2h+\frac{b_{-}}{2}
\bigg)\binom{s}{h}+(h+1)\binom{s}{h+1}\bigg].\] The same considerations we
applied for evaluating the coefficient of $f_{k+2,n,p}^{(\tilde{\mu}^{2})}$ in
Lemma \ref{Ltheta} show that the expression in brackets equals
$\big(s+\frac{b_{-}}{2}\big)\binom{s+1}{h}$. Applying the functional equation
of the gamma function once more, this yields the asserted value of $a_{s+1,h}$.
This completes the proof of the proposition.
\end{proof}

Finally, we come to the
\begin{proof}[Proof of Proposition \ref{Lb-Rtaum}]
We begin by proving that for any $q\in\mathbb{N}$, the action of the operator
$(-4\pi i)^{q}\delta_{1-\frac{b_{-}}{2}+r+t-2l,\tau}^{q}$ sends
$y^{\frac{b_{-}}{2}+2l}\Theta_{L,r,s,t}^{(l)}(\tau,Z)$ to
\[\sum_{h=0}^{q}\binom{q}{h}(4\pi^{2})^{h}y^{\frac{b_{-}}{2}+2l-2q+2h}\frac{
l!\Gamma\big(l+\frac{b_{-}}{2}\big)}{(l-q+h)!\Gamma\big(l-q+h+\frac{b_{-}}{2}
\big)}\Theta_{L,r+h,s+h,t+h}^{(l-q+h)}(\tau,Z).\] For $q=0$ the assertion is
trivially true. We write the asserted function of $y$ preceding the theta
function in the term corresponding to $h$ in the sum arising from the index $q$
as $\binom{q}{h}(4\pi^{2})^{h}b_{l-q+h}(y)$. Given that this assertion holds for
$q$, we apply Equation \eqref{deltakTheta} for the operator $-4\pi
i\delta_{1-\frac{b_{-}}{2}+r+t-2l+2q}$ acting on each term, and observe that the
resulting theta functions correspond to the index $q+1$ and to the summation
indices $h+1$ and $h$. Moreover, after the usual index change manipulations one
sees that the total coefficient in front of the theta function with indices
$q+1$ and $h$ is
\[(4\pi^{2})^{h}\bigg[\binom{q}{h-1}b_{l-q+h-1}(y)+\binom{q}{h}
(l-q+h)\bigg(l-q+h+\frac{b_{-}}{2}-1\bigg)\frac{b_{l-q+h-1}(y)}{y^{2}}\bigg].\]
As the second term here is easily seen to be just
$\binom{q}{h-1}b_{l-q+h-1}(y)$, the inductive assertion follows from the
classical property of the binomial coefficients. With $r=s=0$ and $t=l=q=m$ the
general formula from above becomes
\[\sum_{h}\binom{m}{h}\frac{\Gamma\big(m+\frac{b_{-}}{2}\big)}{
\Gamma\big(h+\frac{b_{-}}{2}\big)}\frac{m!(4\pi^{2}y^{2})^{h}}{h!}y^{\frac{b_{-}
}{2}}\Theta_{L,h,h,m+h}^{(h)}(\tau,Z).\] On the other hand, Putting $m=s$ in
Proposition \ref{LsTheta}, multiplying by $y^{\frac{b_{-}}{2}}$ (which commutes
with differential operators in the variable $Z$), and taking the complex
conjugate of the result, yields precisely the same expression. This proves the
proposition.
\end{proof}

\noindent\textsc{Einstein Institute of Mathematics, the Hebrew University of Jerusalem, Edmund Safra Campus, Jerusalem 91904, Israel}

\noindent E-mail address: zemels@math.huji.ac.il

\end{document}